\documentclass[11pt]{amsart}
\usepackage{amsfonts,amssymb,amsmath,amsthm}
\usepackage{enumitem}
\usepackage[pdfusetitle,colorlinks,citecolor=blue,urlcolor=black,linkcolor=black]{hyperref}

\makeatletter

\newtheorem{thm}{Theorem}[section]
\newtheorem{mainthm}{Theorem}

\newtheorem{lemma}[thm]{Lemma}
\newtheorem{cor}[thm]{Corollary}
\newtheorem{claim}{Claim}[thm]
\newtheorem{prop}[thm]{Proposition}
\newtheorem{fact}[thm]{Fact}

\newtheorem{subclaim}{Subclaim}[claim]

\theoremstyle{definition}
\newtheorem{defn}[thm]{Definition}
\newtheorem{question}[thm]{Question}
\newtheorem*{DEF}{Definition}

\theoremstyle{remark}
\newtheorem{remark}[thm]{Remark}
\newtheorem{conv}[thm]{Convention}

\newcommand\symdiff{\mathbin{\bigtriangleup}}
\renewcommand{\mid}{\mathrel{|}\allowbreak}
\renewcommand{\restriction}{\mathbin\upharpoonright}

\DeclareMathOperator{\acc}{acc}
\DeclareMathOperator{\ch}{CH}
\DeclareMathOperator{\vspec}{Vspec}
\DeclareMathOperator{\ad}{AD}
\DeclareMathOperator{\ssup}{ssup}
\DeclareMathOperator{\reg}{Reg}
\DeclareMathOperator{\nacc}{nacc}

\DeclareMathOperator{\cf}{cf}
\DeclareMathOperator{\dom}{dom}
\DeclareMathOperator{\im}{Im}
\DeclareMathOperator{\otp}{otp}

\DeclareMathOperator{\h}{ht}
\DeclareMathOperator{\tr}{Tr}
\DeclareMathOperator{\p}{P}

\newcommand\ind{\textup{ind}}
\newcommand*\cvec[1]{\vec{\mathcal#1}}
\newcommand\s{\subseteq}
\newcommand\sq{\sqsubseteq}
\newcommand\sqleft[1]{\mathrel{_{#1}{\sqsubseteq}}}
\newcommand\sqleftup[1]{\mathrel{^{#1}{\sqsubseteq}}}

\def\s{\subseteq}
\def\br{\blacktriangleright}

\newcommand*\axiomfont[1]{\textsf{\textup{#1}}}

\newcommand\gch{\axiomfont{GCH}}
\newcommand\pfa{\axiomfont{PFA}}
\newcommand\stp{\axiomfont{STP}}
\def\sd{\framebox[3.0mm][l]{$\diamondsuit$}\hspace{0.5mm}{}}
\providecommand{\myceil}[1]{{}^\ulcorner #1{}^\urcorner }

\title{The vanishing levels of a tree}

\author{Assaf Rinot}
\address{Department of Mathematics, Bar-Ilan University, Ramat-Gan 52900, Israel.}
\urladdr{http://www.assafrinot.com}

\author{Shira Yadai}
\address{Department of Mathematics, Bar-Ilan University, Ramat-Gan 52900, Israel.}
\email{greenss@biu.ac.il}

\author{Zhixing You}
\address{Department of Mathematics, Bar-Ilan University, Ramat-Gan 52900, Israel.}
\email{zhixingy121@gmail.com}

\keywords{Vanishing levels, subtle tree property, coherent tree, regressive tree, almost-disjoint ladder system, proxy principle, C-sequence number}
\subjclass[2010]{Primary 03E05. Secondary 03E35, 03E55}

\begin{document}
\begin{abstract}
We initiate the study of the spectrum $\vspec(\kappa)$ of sets that can be realized as
the vanishing levels $V(\mathbf T)$ of a normal $\kappa$-tree $\mathbf T$.
The latter is an invariant in the sense that if $\mathbf T$ and $\mathbf T'$ are club-isomorphic,
then $V(\mathbf T)\symdiff V(\mathbf T')$ is nonstationary.
Additional features of this invariant imply that $\vspec(\kappa)$ is closed under finite unions and intersections.

The set $V(\mathbf T)$ must be stationary for an homogeneous normal $\kappa$-Aronszajn tree $\mathbf T$,
and if there exists a special $\kappa$-Aronszajn tree, then there exists one $\mathbf T$ that is homogeneous and satisfies $V(\mathbf T)=\kappa$ (modulo clubs).
It is consistent (from large cardinals) that there is an $\aleph_2$-Souslin tree, and yet $V(\mathbf T)$ is co-stationary for every $\aleph_2$-tree $\mathbf T$.
Both $V(\mathbf T)=\emptyset$ and $V(\mathbf T)=\kappa$ (modulo clubs) are shown to be feasible using $\kappa$-Souslin trees, even at some large cardinal close to a weakly compact.
It is also possible to have a family of $2^\kappa$ many $\kappa$-Souslin trees
for which the corresponding family of vanishing levels forms an antichain modulo clubs.

\end{abstract}
\date{Preprint as of September 26, 2023. For updates, visit \textsf{http://p.assafrinot.com/58}.}
\maketitle

\section{Introduction}

Throughout this paper, $\kappa$ denotes a regular uncountable cardinal.
Recall that a poset $\mathbf T=(T,{<_T})$ is a \emph{$\kappa$-tree} iff all of the following hold:
\begin{enumerate}
\item For every $x\in T$, the set $x_\downarrow:=\{ y\in T\mid y<_T x\}$ is well-ordered by $<_T$.
Hereafter, write $\h(x):=\otp(x_\downarrow,<_T)$;
\item For every ordinal $\alpha<\kappa$, the set $T_\alpha:=\{ x\in T\mid \h(x)=\alpha\}$ is nonempty and has size less than $\kappa$,
and the set $T_\kappa$ is empty.
\end{enumerate}

A subset $B\s T$ is an \emph{$\alpha$-branch} iff $(B,<_T)$ is linearly ordered and $\{\h(x)\mid x\in B\}=\alpha$;
it is said to be \emph{vanishing} iff it has no upper bound in $\mathbf T$.

\begin{DEF}[Vanishing levels]
For a $\kappa$-tree $\mathbf T=(T,<_T)$, let $ V(\mathbf T)$ denote the set of all $\alpha\in\acc(\kappa)$ such that for any $x\in T$ with $\h(x)<\alpha$ there exists a vanishing $\alpha$-branch containing $x$.\footnote{The definition of $\acc(\kappa)$ may be found in Subsection~\ref{nocon} below.}
\end{DEF}

The above is an invariant of trees in the sense that if two $\kappa$-trees $\mathbf T,\mathbf T'$ are isomorphic on a club,
then $V(\mathbf T)$ is equal to $V(\mathbf T')$ modulo a club.
It also satisfies that $V(\mathbf T\otimes \mathbf T')=V(\mathbf T)\cup V(\mathbf T')$
and $V(\mathbf T+\mathbf T')=V(\mathbf T)\cap V(\mathbf T')$ for any two normal $\kappa$-trees $\mathbf T,\mathbf T'$.

The importance of this invariant became apparent in \cite{paper48}, where it was shown that if $\mathbf T$ is a $\kappa$-Souslin tree, i.e., a $\kappa$-tree with no $\kappa$-branches and no $\kappa$-sized antichains,
then the combinatorial principle $\clubsuit_{\ad}(S)$ holds for some subset $S\s\kappa$ that is equal to $V(\mathbf T)$ modulo a club.\footnote{The definition of $\clubsuit_{\ad}$ may be found in the paper's Appendix.}
In particular, if $V(\mathbf T)$ is stationary, then a nontrivial instance of $\clubsuit_{\ad}$ holds true,
and this has important applications in set-theoretic topology.

Surprisingly enough, the first main result of this paper shows that $V(\mathbf T)$ need not be stationary.
This is demonstrated in G\"odel's constructible universe, $\mathsf{L}$, where we obtain the following characterization:
\begin{mainthm}\label{thma} In $\mathsf{L}$, for every (regular uncountable cardinal) $\kappa$ that is not weakly compact, the following are equivalent:
\begin{itemize}
\item there exists a $\kappa$-Souslin tree $\mathbf T$ such that $V(\mathbf T)=\emptyset$;
\item there exists a normal and splitting $\kappa$-tree $\mathbf T$ such that $V(\mathbf T)=\emptyset$;
\item $\kappa$ is not the successor of a cardinal of countable cofinality.
\end{itemize}
\end{mainthm}

On the other extreme, it is possible to have a $\kappa$-Souslin tree $\mathbf T$ with $V(\mathbf T)$ as large as possible.
Again, we obtain a complete characterization:
\begin{mainthm}\label{thmb}
In $\mathsf{L}$, for every (regular uncountable cardinal) $\kappa$ that is not weakly compact, the following are equivalent:
\begin{itemize}
\item there exists a $\kappa$-Souslin tree $\mathbf T$ such that $V(\mathbf T)=\acc(\kappa)$;
\item there exists a $\kappa$-tree $\mathbf T$ such that $V(\mathbf T)=\acc(\kappa)$;
\item $\kappa$ is not subtle.
\end{itemize}
\end{mainthm}

An interesting feature of the proof of Theorem~\ref{thmb} is that it goes through a pump-up theorem
generating $\kappa$-Souslin trees from other input trees with weaker properties.	For a $\kappa$-tree $\mathbf T$,
let $V^-(\mathbf T)$ denote the set of all $\alpha\in\acc(\kappa)$ such that there exists a vanishing $\alpha$-branch.
If $\mathbf T$ is homogeneous, then $V^-(\mathbf T)$ coincides with $V(\mathbf T)$, but in contrast with Theorem~\ref{thma},
for every normal $\kappa$-Aronszajn tree $\mathbf T$, the set $V^-(\mathbf T)$ is necessarily stationary.\footnote{Note that any $\kappa$-Souslin must be normal on a tail end.}

Our first pump-up theorem asserts that the existence of a special $\kappa$-Aronszajn tree $\mathbf T$
is equivalent to the existence of one with $V(\mathbf T)=\acc(\kappa)$.
Our second pump-up theorem asserts that for every $\kappa$-tree $\mathbf K$ there exists a $\kappa$-tree $\mathbf T$ such that
$V^-(\mathbf K)\setminus V(\mathbf T)$ is nonstationary.
Our third pump-up theorem asserts that assuming an instance of the proxy principle $\p(\ldots)$ from \cite{paper22},\footnote{See Definitions \ref{proxydef} and \ref{proxydef2} below.} the corresponding tree $\mathbf T$ may moreover be made to be $\kappa$-Souslin:

\begin{mainthm}\label{thmc} Suppose that $\p(\kappa,2,{\sq^*},1)$ holds. Then:
\begin{enumerate}
\item For every $\kappa$-tree $\mathbf K$,
there exists a $\kappa$-Sousin tree $\mathbf T$ such that $V^-(\mathbf K)\setminus V(\mathbf T)$ is nonstationary. In particular:
\item There exists a $\kappa$-Sousin tree $\mathbf T$ such that $V(\mathbf T)$ is stationary.
\end{enumerate}
\end{mainthm}

The preceding addresses the problem of ensuring $V(\mathbf T)$ to cover some stationary set $S$.
The next theorem addresses the dual problem.
Along the way, it provides a cheap way to obtain a family of $2^\kappa$-many $\kappa$-Souslin trees that are not pairwise club-isomorphic.

\begin{mainthm}\label{thmd}  If $\diamondsuit(S)$ holds for some nonreflecting stationary subset $S$ of a strongly inaccessible cardinal $\kappa$,
then there is an almost disjoint family $\mathcal S$ of $2^\kappa$ many stationary subsets of $S$ such that,
for each $S'\in\mathcal S$, there is a $\kappa$-Souslin tree $\mathbf T$ with $V(\mathbf T)=S'$.
\end{mainthm}

Let us now come back to the motivating problem of getting instances of $\clubsuit_{\ad}$.
By \cite[Theorem~2.30]{paper48},
if $\kappa$ is weakly compact, then $\clubsuit_{\ad}(S)$ fails for every $S$ with $\reg(\kappa)\s S\s\kappa$.
This raises the question as to whether $\clubsuit_{\ad}(S)$ may hold over a large subset $S$ of a cardinal $\kappa$ that is close to being weakly compact.
We answer this question in the affirmative:
\begin{mainthm}\label{thme} Assuming the consistency of a weakly compact cardinal, it is consistent that for some strongly inaccessible cardinal $\kappa$
satisfying $\chi(\kappa)=\omega$,\footnote{$\chi(\kappa)$ can be understood as measuring how far $\kappa$ is from being weakly compact; see Definition~\ref{defcnm} below.}
there is a $\kappa$-Souslin tree $\mathbf T$ such that $V(\mathbf T)=\acc(\kappa)$.
\end{mainthm}

In the appendix to this paper, we improve a result from \cite{paper48} concerning the connection between
Ostaszewski's principle $\clubsuit$ and the principle $\clubsuit_{\ad}$.
As a byproduct, we obtain the following unexpected result:

\begin{mainthm}\label{thmf} If $\clubsuit(S)$ holds over a nonreflecting stationary $S\s\kappa$, then there exists a Dowker space of size $\kappa$.
\end{mainthm}

\subsection{Organization of this paper}

In Section~\ref{sect2}, we develop the basic theory of vanishing levels of trees.
It is proved that if $\kappa$ is not a strong limit, then $V^-(\mathbf T)$ is stationary for every normal and splitting $\kappa$-tree $\mathbf T$.
It is proved that for every $\kappa$-tree $\mathbf K$, there exists a $\kappa$-tree $\mathbf T$ such that $V^-(\mathbf K)\setminus V(\mathbf T)$ is nonstationary,
and that the existence of a special $\kappa$-Aronszajn tree $\mathbf T$
is equivalent to the existence of an homogeneous one with $V(\mathbf T)=\acc(\kappa)$.

In Section~\ref{sec4}, we prove Theorem~\ref{thmc} and some variations of it. As a corollary, we get Theorem~\ref{thmb}
and infer that if $\square_\lambda\mathrel{+}\diamondsuit(\lambda^+)$ holds
for an infinite cardinal $\lambda$,
or if $\square(\lambda^+)\mathrel{+}\gch$ holds for a regular uncountable $\lambda$, then there exists a $\lambda^+$-Souslin tree $\mathbf T$ with $V(\mathbf T)=\acc(\lambda^+)$.

In Section~\ref{sect5}, we address the problem of realizing a given nonreflecting stationary subset of $\kappa$ as $V(\mathbf T)$ for some $\kappa$-Souslin tree $\mathbf T$.
The proof of Theorem~\ref{thmd} will be found there.

In Section~\ref{sect6}, we address the problem of
constructing an homogeneous $\kappa$-Souslin tree $\mathbf T$ such that $V(\mathbf T)=\{ \alpha<\kappa\mid \cf(\alpha)\in x\}$ for a prescribed nonempty finite set $x\s\reg(\kappa)$.
In particular, this is shown to be feasible in $\mathsf L$ whenever $\kappa$ is ${<}\max(x)$-inaccessible.
The proof of Theorem~\ref{thma} will be found there.

In Section~\ref{sect7}, we deal with Souslin trees admitting an ascent path. It is proved that for every uncountable cardinal $\lambda$, $\square_\lambda+\gch$ entails
that for every $\mu\in\reg(\cf(\lambda))$
there exists a $\lambda^+$-Souslin tree $\mathbf T$ with a $\mu$-ascent path such that $V(\mathbf T)=\acc(\lambda^+)$.
The proof of Theorem~\ref{thme} will be found there.

Section~\ref{secA} is a short appendix where we improve \cite[Lemma~2.10]{paper48},
from which we obtain the proof of Theorem~\ref{thmf}.

\subsection{Notation and conventions}\label{nocon}
$H_\kappa$ denotes  the collection of all sets of hereditary cardinality less than $\kappa$.
$\reg(\kappa)$ denotes the set of all infinite regular cardinals $<\kappa$.
For $\chi\in\reg(\kappa)$, $E^\kappa_\chi$
denotes the set $\{\alpha < \kappa \mid \cf(\alpha) = \chi\}$, and
$E^\kappa_{\geq \chi}$, $E^\kappa_{<\chi}$, $E^\kappa_{\neq\chi}$, are defined analogously.

For a set of ordinals $C$, we write $\ssup(C) := \sup\{\alpha + 1 \mid\alpha \in C\}$, $\acc^+(C) := \{\alpha < \ssup(C) \mid \sup(C \cap \alpha) = \alpha > 0\}$,
$\acc(C) := C \cap \acc^+(C)$, and $\nacc(C) := C \setminus \acc(C)$.
For a set $S$, we write $[S]^{\chi}$ for $\{A\s S\mid |A|=\chi\}$, and $[S]^{<\chi}$ is defined analogously.
For a set of ordinals $S$, we  identify $[S]^2$ with $\{ (\alpha,\beta)\mid \alpha,\beta\in S, \alpha<\beta\}$,
and we let $\tr(S):=\{ \beta<\ssup(S)\mid \cf(\beta)>\omega\ \&\ S \cap \beta \text{ is stationary in } \beta\}$.

We define four binary relations over sets of ordinals, as follows:
\begin{itemize}
\item $D\sq C$ iff there exists some ordinal $\beta$ such that $D = C \cap \beta$;
\item $D\sq^* C$ iff $D \setminus \varepsilon\sqsubseteq C \setminus \varepsilon$
for some $\varepsilon < \sup(D)$;
\item $D\sqleftup{S} C$ iff $D\sq C$ and $\sup(D)\notin S$;
\item $D \sqleft{\chi} C$ iff $D \sqsubseteq C$ or $\cf(\sup(D))<\chi$.
\end{itemize}

A \emph{list} over a set of ordinals $S$ is a sequence $\vec A=\langle A_\alpha\mid\alpha\in S\rangle$
such that, for each $\alpha\in S$, $A_\alpha$ is a subset of $\alpha$.
It is said to be \emph{thin} if $|\{ A_\alpha\cap\varepsilon\mid \alpha\in S\}|<\ssup(S)$ for every $\varepsilon<\ssup(S)$.
It is said to be \emph{$\xi$-bounded} if $\otp(A_\alpha)\le\xi$ for all $\alpha\in S$.
A \emph{ladder system} over $S$ is a list $\vec A=\langle A_\alpha\mid\alpha\in S\rangle$
such that $\sup(A_\alpha)=\sup(\alpha)$ for every $\alpha\in S$.
It is said to be \emph{almost disjoint}
if $\sup(A_\alpha\cap A_{\alpha'})<\alpha$ for all $\alpha\neq\alpha'$ in $S$.
A \emph{$C$-sequence} over $S$ is a ladder system $\vec C=\langle C_\alpha\mid\alpha\in S\rangle$ such that
each $C_\alpha$ is a closed subset of $\alpha$.
Finally, a (resp.~thin/$\xi$-bounded/almost-disjoint) \emph{$\mathcal C$-sequence} over $S$ is a sequence $\cvec{C}=\langle\mathcal C_\alpha\mid\alpha\in S\rangle$ of nonempty sets
such that every element of $\prod_{\alpha\in S}\mathcal C_\alpha$ is a (resp.~thin/$\xi$-bounded/almost-disjoint) $C$-sequence.

\section{The basic theory of vanishing levels}\label{sect2}
\begin{defn}\label{defn21}
A tree $\mathbf T=(T,<_T)$ is said to be:
\begin{itemize}
\item \emph{Hausdorff} iff for every limit ordinal $\alpha$ and all $x,y\in T_\alpha$,
if $x_\downarrow=y_\downarrow$, then $x=y$;
\item \emph{normal} iff for every pair $\alpha<\beta$ of ordinals, if $T_\beta\neq\emptyset$, then for every $x\in T_\alpha$ there exists $y\in T_\beta$ with $x<_T y$;
\item \emph{$\chi$-complete} iff any $<_T$-increasing sequence of elements of $\mathbf T$, and of length $<\chi$, has an upper bound in $\mathbf T$;
\item \emph{$\varsigma$-splitting} iff every node of $\mathbf T$ admits at least $\varsigma$-many immediate successors,
that is, for every $x\in T$, $|\{ y\in T\mid x<_T y, \h(y)=\h(x)+1\}|\ge\varsigma$.
By \emph{splitting}, we mean $2$-splitting;
\item \emph{$\kappa$-Aronszajn} iff $\mathbf T$ is a $\kappa$-tree with no $\kappa$-branches;
\item \emph{special $\kappa$-Aronszajn tree} iff it is a $\kappa$-Aronszajn and there exists a map $\rho:T\rightarrow T$ satisfying the following:
\begin{itemize}
\item for every non-minimal $x\in T$, $\rho(x)<_T x$;
\item for every $y\in T$, $\rho^{-1}\{y\}$ is covered by less than $\kappa$ many antichains.
\end{itemize}
\end{itemize}
\end{defn}
\begin{remark} All the $\kappa$-Souslin trees constructed in this paper will be Hausdorff, normal and splitting.
\end{remark}

\begin{defn}\label{defn23}	For a $\kappa$-tree $\mathbf T=(T,<_T)$:
\begin{enumerate}
\item $ V^-(\mathbf T)$ denotes the set of all $\alpha\in\acc(\kappa)$ such that there exists a vanishing $\alpha$-branch;
\item $ V(\mathbf T)$ denotes the set of all $\alpha\in\acc(\kappa)$ such that for every $x\in T$ with $\h(x)<\alpha$ there exists a vanishing $\alpha$-branch containing $x$.
\item $\vspec(\kappa):=\{ V(\mathbf T)\mid \mathbf T\text{ is a normal }\kappa\text{-tree}\}$;
\item For $A\s\kappa$, we write $T\restriction A:=\{ x\in T\mid \h(x)\in A\}$.
\end{enumerate}
\end{defn}

Note that if $\mathbf T$ is a $\kappa$-tree such that $V(\mathbf T)$ is cofinal in $\kappa$, then $\mathbf T$ is normal.

\begin{lemma}\label{lemma33} Suppose that $\mathbf T$ is a $\kappa$-tree such that $V^-(\mathbf T)$ (resp.~$V(\mathbf T)$)
covers a club in $\kappa$. Then there exists a subtree $\mathbf T'$ of $\mathbf T$ such that $V^-(\mathbf T)$ (resp.~$V(\mathbf T)$) is equal to $\acc(\kappa)$.
\end{lemma}
\begin{proof} Let $D\s\kappa$ be a club as in the hypothesis.
Then $\mathbf T':=(T\restriction D,<_T)$ is a subtree as sought.
\end{proof}

\begin{prop}\label{prop22} For a $\kappa$-tree $\mathbf T=(T,<_T)$:
\begin{enumerate}
\item If $\mathbf T$ is a normal $\kappa$-Aronszajn tree, then $V^-(\mathbf T)$ is stationary;
\item If $\mathbf T$ is homogeneous,\footnote{That is, for all $\alpha<\kappa$ and $s,t\in T_\alpha$, there is an automorphism of $\mathbf T$ sending $s$ to $t$.} then $ V^-(\mathbf T)=V(\mathbf T)$.
\end{enumerate}
\end{prop}
\begin{proof}
(1) Suppose not, and fix a club $D\s\kappa$ disjoint from $V^{-}(\mathbf T)$.
We shall construct a $<_T$-increasing sequence $\langle t_\alpha\mid\alpha\in D\rangle$ in such a way that $t_\alpha\in T_\alpha$ for all $\alpha\in D$,
contradicting the fact that $\mathbf T$ is $\kappa$-Aronszajn.
We start by letting $t_{\min(D)}$ be an arbitrary element of $T_{\min(D)}$. Next, for every $\alpha\in D$ such that $t_\alpha$ has already been successfully defined,
we set $\beta:=\min(D\setminus(\alpha+1))$,
and use the normality of $\mathbf T$ to pick $t_\beta$ in $T_\beta$ extending $t_\alpha$.
For every $\alpha\in\acc(D)$ such that $\langle t_\epsilon\mid\epsilon\in D\cap\alpha\rangle$ has already been defined,
the latter clearly induces an $\alpha$-branch,
so the fact that $\alpha\notin V^-(\mathbf T)$ implies that there exists some $t_\alpha\in T_\alpha$
such that $t_\epsilon<_T t_\alpha$ for all $\epsilon\in D\cap\alpha$. This completes the description of the recursion.

(2) Suppose that $\mathbf T$ is homogeneous. Let $\alpha\in V^-(\mathbf T)$,
and fix a vanishing $\alpha$-branch $b$.
Now, given a node $x$ of $\mathbf T$ of height less than $\alpha$,
let $y$ be the unique element of $b$ to have the same height as $x$.
Since $\mathbf T$ is homogeneous, there exists an automorphism $\pi$ of $\mathbf T$ sending $y$ to $x$,
and it is clearly the case that $\pi[b]$ is a vanishing $\alpha$-branch through $x$.
\end{proof}

\begin{prop} If $\square(\kappa)$ holds, then there exists a $\kappa$-Aronszajn tree $\mathbf T$ such that $V(\mathbf T)=E^\kappa_\omega$.
\end{prop}
\begin{proof} By \cite[Theorem~3.9]{MR2013395}, $\square(\kappa)$ yields a sequence of functions $\langle f_\beta:\beta\rightarrow\beta\mid\beta\in\acc(\kappa)\rangle$
such that:
\begin{itemize}
\item for every $(\beta,\gamma)\in[\acc(\kappa)]^2$, $\{\alpha<\beta\mid f_\beta(\alpha)\neq f_\gamma(\alpha)\}$ is finite;
\item there is no cofinal $B\s\acc(\kappa)$ such that $\{ f_\beta\mid \beta\in B\}$ is linearly ordered by $\s$.
\end{itemize}
Set $T:=\{ f\in{}^\alpha\alpha\mid \alpha\le\beta<\kappa, f\text{ disagrees with }f_\beta\text{ on a finite set}\}$.
Then $\mathbf T=(T,{\s})$ is a uniformly coherent $\kappa$-Aronszajn tree.
By \cite[Remark~2.20]{paper48}, then, $V(\mathbf T)=E^\kappa_\omega$.
\end{proof}

\begin{defn}\label{regmap} For a $\kappa$-tree $\mathbf T=(T,<_T)$ and a subset $S\s\kappa$,
we say that $\mathbf T$ is \emph{$S$-regressive} iff there exists a map $\rho:T\restriction S\rightarrow T$ satisfying the following:
\begin{itemize}
\item for every $x\in T\restriction S$, $\rho(x)<_T x$;
\item for all $\alpha\in S$ and $x,y\in T_\alpha$,
if $\rho(x)<_T y$ and $\rho(y)<_T x$, then $x=y$.
\end{itemize}
\end{defn}
\begin{remark}\label{rmk25} If $\rho$ is as above, then every map $\varrho:T\restriction S\rightarrow T$
satisfying $\rho(x)\le_T \varrho(x)<_T x$ for all $x\in T\restriction S$ is as well a witness
to $\mathbf T$ being $S$-regressive.
\end{remark}

The next lemma generalizes \cite[Lemmas 2.19 and 2.21]{paper48}.

\begin{lemma}\label{cor53}  Suppose that:
\begin{itemize}
\item $\mathbf T$ is a normal, $\varsigma$-splitting $\kappa$-tree, for some fixed cardinal $\varsigma<\kappa$;
\item $S\s E^\kappa_\chi$ is stationary for some fixed regular cardinal $\chi<\kappa$;
\item Either of the following:
\begin{enumerate}
\item $\varsigma^\chi\ge\kappa$;
\item $T$ is $S$-regressive and $\varsigma^{<\chi}<\varsigma^\chi$;
\item $T$ is $S$-regressive, $\chi=\varsigma$ and there exists a weak $\chi$-Kurepa tree.\footnote{That is, a tree of height and size $\chi$ admitting at least $\chi^+$-many branches.}
\end{enumerate}
\end{itemize}

Then, for every $\alpha\in S$, either $\alpha\in V(\mathbf T)$
or ($\cf(\alpha)>\omega$ and) $V^-(\mathbf T)\cap\alpha$ is stationary in $\alpha$.
In particular, $V^-(\mathbf T)\cap E^\kappa_{\le\chi}$ is stationary.
\end{lemma}
\begin{proof} Write $\mathbf T=(T,{<_T})$. Towards a contradiction, suppose that $\alpha\in S$ is a counterexample.
As $\alpha\notin V(\mathbf T)$,
we may fix $x\in T$ with $\h(x)<\alpha$ such that every $\alpha$-branch $B$ with $x\in B$ has an upper bound in $\mathbf T$.
Since either $\cf(\alpha)\le\omega$ or $V^-(\mathbf T)\cap\alpha$ is nonstationary in $\alpha$,
we may fix a club $C$ in $\alpha$ of order-type $\chi$ such that $\min(C)=\h(x)$ and such that $\acc(C)\cap V^-(\mathbf T)=\emptyset$.

Let $\langle \alpha_i\mid i<\chi\rangle$ denote the increasing enumeration of $C$.
We shall recursively construct an array of nodes $\langle t_s\mid s\in{}^{<\chi}\varsigma\rangle$ in such a way that $t_s\in T_{\alpha_{\dom(s)}}$.
Set $t_\emptyset:=x$. For every $i<\chi$ and every $s:i\rightarrow\varsigma$ such that $t_s$ has already been defined, since $T$ is normal and $\varsigma$-splitting,
we may find an injective sequence $\langle t_{s{}^\smallfrown\langle j\rangle}\mid j<\varsigma\rangle$ of nodes of $T_{\alpha_{i+1}}$ all extending $t_s$.
For every $i\in\acc(\chi)$ such that $\langle t_s\mid s\in{}^{<i}\varsigma\rangle$ has already been defined,
for every $s:i\rightarrow\varsigma$, since $\{ t_{s\restriction\iota}\mid \iota<i\}$ induces an $\alpha_i$-branch,
the fact that $\alpha_i\notin V^-(\mathbf T)$ implies that we may find $t_s\in T_{\alpha_i}$ that is a limit of that $\alpha_i$-branch. This completes the recursive construction of our array.

For every $s\in{}^{\chi}\varsigma$, $B_s:=\{ t\in T\mid \exists i<\chi\,(t<_T t_{s\restriction i})\}$ is an $\alpha$-branch containing $x$,
and hence there must be some $b_s\in T_\alpha$ extending all elements of $B_s$.
Our construction also ensures that $B_s\neq B_{s'}$ whenever $s\neq s'$. We now consider a few options:
\begin{enumerate}
\item Suppose that $\varsigma^\chi\ge\kappa$. Then $|T_\alpha|\ge|\{ b_s\mid s\in{}^\chi\varsigma\}|=\varsigma^\chi\ge\kappa$. This is a contradiction.
\item Suppose that $\mathbf T$ is $S$-regressive, as witnessed by $\rho:T\restriction S\rightarrow T$.
For every $s\in{}^{\chi}\varsigma$, $\rho(b_s)$ belongs to $B_s$, but by Remark~\ref{rmk25},
we may assume that $\rho(b_s)= t_{s\restriction i}$ for some $i<\chi$.
\begin{itemize}
\item[$\br$] If $\varsigma^{<\chi}<\varsigma^\chi$, then we may now find $s\neq s'$ in ${}^\chi\varsigma$ such that $\rho(b_s)=\rho(b_{s'})$.
Then, $\rho(b_{s'})<_T t_s$ and $\rho(b_s)<_T t_{s'}$,
contradicting the fact that $b_s\neq b_{s'}$.
\item[$\br$] If $\chi=\varsigma$ and there exists a weak $\chi$-Kurepa tree,
then this may be witnessed by a tree of the form $(K,{\s})$ for some $K\s {}^{<\chi}\varsigma$.
Let $\langle s_\beta\mid\beta<\chi^+\rangle$ be an injective enumeration of branches through $(K,{\s})$.
Since $|K|\le\chi$, there must exist $\beta\neq\beta'$ such that $\rho(b_{s_\beta})=\rho(b_{s_{\beta'}})$,
which yields a contradiction as in the previous case.
\qedhere
\end{itemize}

\end{enumerate}
\end{proof}

\begin{cor}\label{cor26} If $\kappa$ is not a strong limit, then for every normal and splitting $\kappa$-tree $\mathbf T$,
$V^-(\mathbf T)$ is stationary.
\end{cor}
\begin{proof} Suppose that $\kappa$ is not a strong limit. It is not hard to see that there exists some infinite cardinal $\varsigma<\kappa$ for which there exists a regular cardinal $\chi<\kappa$ such that $\varsigma^\chi\ge\kappa$.
Now, given a normal and splitting $\kappa$-tree $\mathbf T=(T,<_T)$,
as shown in the proof of \cite[Proposition~2.16]{paper48},
the club $D:=\{\alpha<\kappa\mid \alpha=\varsigma^\alpha\}$ satisfies that $\mathbf T'=(T\restriction D,{<_T})$ is normal and $\varsigma$-splitting.
By Lemma~\ref{cor53}, $V^-(\mathbf T')$ is stationary. As $D$ is a club in $\kappa$, this means that $V^-(\mathbf T)$ is stationary, as well.
\end{proof}
\begin{cor}\label{cor27} If $\kappa=\lambda^+$ is a successor cardinal and $\lambda^{\aleph_0}\ge\kappa$, then for every normal and splitting $\kappa$-tree $\mathbf T$,
$E^\kappa_\omega\setminus V(\mathbf T)$ is nonstationary.
\end{cor}
\begin{proof} Suppose that $\kappa$ and $\lambda$ are as above.
Now, given a normal and splitting $\kappa$-tree $\mathbf T=(T,<_T)$,
the club $D:=\{\alpha<\kappa\mid \alpha=\lambda^\alpha\}$ satisfies that $\mathbf T'=(T\restriction D,{<_T})$ is normal and $\lambda$-splitting.
By Lemma~\ref{cor53}, $V(\mathbf T')\supseteq E^\kappa_\omega$. As $D$ is a club in $\kappa$, this means that $E^\kappa_\omega\setminus V(\mathbf T)$ is nonstationary.
\end{proof}

\begin{defn}[\cite{paper23}]\label{Streamlined}
A \emph{streamlined $\kappa$-tree} is a subset $T\s{}^{<\kappa}H_\kappa$
such that the following two conditions are satisfied:
\begin{enumerate}
\item $T$ is downward-closed, i.e, for every $t\in T$, $\{ t\restriction \alpha\mid \alpha<\kappa\}\s T$;
\item for every $\alpha<\kappa$, the set
$T_\alpha:=T\cap{}^\alpha\kappa$ is nonempty and has size $<\kappa$.
\end{enumerate}
For every $\alpha\le\kappa$, we denote
$\mathcal B(T\restriction\alpha):=\{f\in{}^\alpha H_\kappa\mid\forall\beta<\alpha\,(f\restriction\beta\in T)\}$.
\end{defn}

Note that every streamlined tree is Hausdorff.
\begin{conv} We identify a streamlined tree $T$ with the poset $\mathbf T=(T,{\s})$.
\end{conv}

\begin{defn}\label{regdef}
For two elements $s,t$ of $H_\kappa$, we define $s*t$ to be the emptyset,
unless $s,t\in{}^{<\kappa}H_\kappa$ with $\dom(s)\le\dom(t)$, in which case $s*t:\dom(t)\rightarrow H_\kappa$ is defined by stipulating:
$$(s*t)(\beta):=\begin{cases}s(\beta),&\text{if }\beta\in\dom(s);\\
t(\beta),&\text{otherwise.}\end{cases}$$
\end{defn}
\begin{defn} A streamlined $\kappa$-tree $T$  is
\emph{uniformly homogeneous} iff for all $\alpha<\beta<\kappa$,
$s\in T_\alpha$ and $t\in T_\beta$, $s*t$ is in $T$.
\end{defn}

The next proposition should be clear, but we include a proof sketch.

\begin{prop}\label{prop53} Suppose that $T$ is a streamlined $\kappa$-tree that is uniformly homogeneous.
Then $T$ is indeed homogeneous.
\end{prop}
\begin{proof} Let $\alpha<\kappa$ and $s,s'\in T_\alpha$. Define $\pi:T\rightarrow T$ via:
$$\pi(t):=\begin{cases}
s'\restriction \dom(t),&\text{if }t\s s;\\
s\restriction \dom(t),&\text{if }t\s s';\\
s'*t,&\text{if }t\supseteq s;\\
s*t,&\text{if }t\supseteq s';\\
t,&\text{otherwise}.
\end{cases}$$
Then $\pi$ is a well-defined automorphism of $T$,
sending $s$ to $s'$.
\end{proof}

\begin{lemma}\label{lemma311} For a stationary $S\s\kappa$, the following are equivalent:
\begin{enumerate}
\item There exist a club $D\s\kappa$ and a thin ladder system $\langle A_\alpha\mid\alpha\in S\cap D\rangle$
such that, for every $(\alpha,\beta)\in [S\cap D]^2$, $\sup(A_\alpha\cap A_\beta)<\alpha$;
\item There exist a club $D\s\kappa$ and a thin ladder system $\langle A_\alpha\mid\alpha\in S\cap D\rangle$
such that, for every $(\alpha,\beta)\in [S\cap D]^2$, $A_\alpha\neq A_\beta\cap\alpha$;
\item There exist a club $D\s\kappa$ and a uniformly homogeneous streamlined $\kappa$-tree
$T$ such that $V(T)\supseteq S\cap D$;
\item There exist a club $D\s\kappa$ and a $\kappa$-tree
$\mathbf T$ such that $V^-(\mathbf T)\supseteq S\cap D$.
\end{enumerate}
\end{lemma}
\begin{proof} $(1)\implies(2)$: This is immediate.

$(2)\implies(3)$: Suppose that $D$ and $\langle A_\alpha\mid\alpha\in S\cap D\rangle$ are as in (2).
Let $\langle x_i\mid i<\kappa\rangle$ be an injective enumeration of $\langle A_\alpha\cap\varepsilon\mid \varepsilon<\alpha, \alpha\in S\cap D\rangle$.
For each $\alpha\in S\cap D$, let $k_\alpha:\alpha\rightarrow\kappa$ be the unique function to satisfy for all $\varepsilon<\alpha$:
$$A_\alpha\cap\varepsilon=x_{k_\alpha(\varepsilon)}.$$

Define first an auxiliary collection $K$ by letting $$K:=\{ k_\beta\restriction\alpha\mid \alpha<\beta, \beta\in S\cap D\}.$$
Note that $\{ \dom(y)\mid y\in K\}=\kappa$ and that $K$ is closed under taking initial segments.
So $K$ is a streamlined $\kappa$-tree because otherwise there must exist some $\varepsilon<\kappa$ such that
$\{ k_\beta\restriction\varepsilon\mid \beta\in S\cap D\}$ has size $\kappa$,
contradicting the fact that $\langle A_\beta\mid\beta\in S\cap D\rangle$ is thin.
We shall use $K$ to construct a uniformly homogeneous streamlined $\kappa$-tree $T$ by defining its levels $T_\alpha$ by recursion on $\alpha<\kappa$.

Start by letting $T_0:=K_0$. Clearly, $T_0=\{\emptyset\}$, so that $|T_0|<\kappa$. Next, for every nonzero $\alpha<\kappa$ such that $T\restriction\alpha$ has already been defined and have size less than $\kappa$, let
$$T_{\alpha}:=\{ x*y\mid x\in T\restriction\alpha,\, y\in K_\alpha\}$$
and note that $|T_\alpha|<\kappa$.
Altogether, $T$ is a streamlined $\kappa$-tree.
\begin{claim}\label{c2171} $T$ is uniformly homogeneous.
\end{claim}
\begin{proof} We prove that $x*y\in T$
for all $x,y\in T$ with $\dom(x)<\dom(y)$.
The proof is by induction on $\dom(y)$.
So suppose that $\alpha<\kappa$ is such that
for all $x,y\in T$ with $\dom(x)<\dom(y)<\alpha$, it is the case that $x*y\in T$, and let $x,y\in T$ with $\dom(x)<\dom(y)=\alpha$.
Recalling the definition of $T_\alpha$, pick $x'\in T\restriction\alpha$ and $y'\in K_\alpha$ such that $y=x'*y'$.

$\br$ If $\dom(x)<\dom(x')$, then $x*y=x*(x'*y')=(x*x')*y'$. As $\dom(x)<\dom(x')<\alpha$, the induction hypothesis implies that $x*x'\in T\restriction\alpha$,
and then the definition of $T_\alpha$ implies that $(x*x')*y'$ is in $T$.

$\br$ If $\dom(x)\ge\dom(x')$, then $x*y=x*(x'*y')=x*y'$,
and then the definition of $T_\alpha$ implies that $x*y'$ is in $T$.
\end{proof}

By the preceding claim together with Proposition~\ref{prop22}, it now suffices to prove that $V^{-}(T)\supseteq S\cap D\cap\acc(\kappa)$.
To this end, let $\alpha\in S\cap D\cap\acc(\kappa)$.
Clearly, $b:=\{ k_\alpha\restriction\varepsilon\mid\varepsilon<\alpha\}$ is an $\alpha$-branch in $K$ and hence in $T$.
If $b$ is not vanishing in $T$, then we may find $x\in T\restriction\alpha$ and $y\in K_\alpha$ such that $x*y=k_\alpha$.
Recalling the definition of $K_\alpha$,
we may pick $\beta\in S\cap D$ above $\alpha$ such that $y=k_\beta\restriction\alpha$.
As $\alpha<\beta$, it is the case that $A_\alpha\neq A_{\beta}\cap\alpha$,
so we may pick $\delta\in A_\alpha\Delta(A_{\beta}\cap\alpha)$.
Then $\varepsilon:=\max\{\delta,\dom(x)\}+1$ is smaller than $\alpha$ and satisfies $k_\alpha(\varepsilon)\neq k_{\beta}(\varepsilon)$,
contradicting the fact that $k_\alpha(\varepsilon)=(x*y)(\varepsilon)=y(\varepsilon)=k_\beta(\varepsilon)$.

$(3)\implies(4)$: This is immediate.

$(4)\implies(1)$ Every $\kappa$-tree is order-isomorphic to an ordinal-based tree (see, e.g., \cite[Proposition~2.16]{paper48}),
so we may assume that we are given a tree $\mathbf T$ of the form $(\kappa,<_T)$
and a club $D\s\kappa$ such that $V^-(\mathbf T)\supseteq S\cap D$.
By possibly shrinking $D$, we may also assume that $D\s\acc\{\beta<\kappa\mid T\restriction\beta=\beta\}$.
It follows that for every $\alpha\in D$, every $\alpha$-branch is a cofinal subset of $\alpha$.
For every $\alpha\in S\cap D$, let $A_\alpha$ be a vanishing $\alpha$-branch. As $\mathbf T$ is a $\kappa$-tree,
the ladder system $\langle A_\alpha\mid\alpha\in S\cap D\rangle$ is thin.
In addition, for every $(\alpha,\beta)\in [S\cap D]^2$, if it were the case that $\sup(A_\beta\cap A_\alpha)=\alpha$,
then  $\min(A_\beta\setminus A_\alpha)$ is a node extending all elements of $A_\alpha$,
contradicting the fact that $A_\alpha$ is vanishing. So, $\sup(A_\beta\cap A_\alpha)<\alpha$.
\end{proof}

When $S$ is a club, the preceding is related to the subtle tree property:
\begin{defn}[Weiß, \cite{Chris10}]
$\kappa$ has the \emph{subtle tree property} ($\kappa$-$\stp$ for short) iff
for every thin list $\langle A_\alpha \mid \alpha\in D\rangle$
over a club $D \subseteq \kappa$,
there exists a pair $(\alpha,\beta)\in [D]^2$ such that $A_\alpha=A_\beta \cap \alpha$.
\end{defn}

\begin{cor}\label{stp} All of the following are equivalent:
\begin{itemize}
\item $\kappa$-$\stp$ fails;
\item there is a $\kappa$-tree $\mathbf T$ with $V^-(\mathbf T)=\acc(\kappa)$;
\item there is an homogeneous $\kappa$-tree $\mathbf T$ with $V(\mathbf T)=\acc(\kappa)$;
\item there is a uniformly homogeneous streamlined $\kappa$-tree $T$ such that $V(T)$ covers a club in $\kappa$.
\end{itemize}
\end{cor}
\begin{proof}  By Lemmas \ref{lemma311} and \ref{lemma33}.
\end{proof}
\begin{remark} By \cite[Theorem~3.2.5]{Chris10}, $\pfa$ implies that $\aleph_2$-$\stp$ holds.
By \cite[Theorem~1.2]{HS20},
if $\lambda$ is the singular limit of supercompact cardinals
then $\lambda^+$-$\stp$ fails.\footnote{The statement of the theorem in \cite{HS20} is limited to countable cofinality, but the proof works unconditionally.}
\end{remark}

\begin{cor} Assuming the consistency of a subtle cardinal, it is consistent that the conjunction of the following holds true:
\begin{itemize}
\item there exists an $\aleph_2$-Souslin tree;
\item for every normal and splitting $\aleph_2$-tree $\mathbf T$, $E^{\aleph_2}_{\aleph_1}\setminus V(\mathbf T)$ is stationary.
\end{itemize}
\end{cor}
\begin{proof} Fix a subtle cardinal $\kappa$ that is not weakly compact in $\mathsf L$,
and work in the forcing extension by Mitchell's forcing of length $\kappa$.
By \cite[Theorem~2.3.1]{Chris10}, $\aleph_2$-$\stp$ holds,
and hence $V(\mathbf T)$ cannot contain a club for every $\aleph_2$-tree $\mathbf T$.
In addition, this is a model in which $2^{\aleph_0}=\aleph_2$
and hence Corollary~\ref{cor27} implies that $E^{\aleph_2}_{\aleph_0}\setminus V(\mathbf T)$ is nonstationary
for every normal and splitting $\aleph_2$-tree $\mathbf T$.
Therefore, $E^{\aleph_2}_{\aleph_1}\setminus V(\mathbf T)$ is stationary
for every normal and splitting $\aleph_2$-tree $\mathbf T$.
In addition, this is a model in which $\mathfrak b=\aleph_1$, $2^{\aleph_1}=\aleph_2$, and (since $\kappa$ is not weakly compact in $\mathsf{L}$) $\square(\aleph_2)$ holds.
So, by \cite[Theorem~A]{paper51}, there exists an $\aleph_2$-Souslin tree.
\end{proof}

\begin{cor}\label{cor213} Suppose that $S$ is a stationary subset of a strongly inaccessible $\kappa$.
Then there exists a $\kappa$-tree $\mathbf T$ such that $V(\mathbf T)\cap S$ is stationary.
\end{cor}
\begin{proof} By Lemma~\ref{lemma311}, it suffices to find a stationary $S^-\s S$ that carries a thin almost disjoint $C$-sequence. We consider two cases:

$\br$ If $S\cap E^\kappa_\omega$ is stationary, then set $S^-:=S\cap E^\kappa_\omega$,
and let $\langle C_\alpha\mid\alpha\in S^-\rangle$ be some $\omega$-bounded $C$-sequence over $S^-$.

$\br$ Otherwise, let $S^-:=S\setminus(E^\kappa_\omega\cup\tr(S))$. Then $S^-$ is stationary,
and for every $\alpha\in S^-$, we may pick a club $C_\alpha$ in $\alpha$ that is disjoint from $S$.
Evidently, $\sup(C_{\alpha'}\cap C_\alpha)<\alpha'$ for every $(\alpha,\alpha')\in[S^-]^2$.
\end{proof}

\begin{lemma}\label{l22} If $\theta\in\reg(\kappa)$ is such that $\lambda^{<\theta}<\kappa$ for all $\lambda<\kappa$,
then there exists an almost disjoint thin $C$-sequence over $E^\kappa_\theta$.
\end{lemma}
\begin{proof} Just take a $\theta$-bounded $C$-sequence over $E^\kappa_\theta$.
\end{proof}

Building on the work of Todor\v{c}evi\'c \cite{MR2355670} and Krueger \cite{MR3078820}, we obtain the following pump-up theorem for special $\kappa$-Aronszajn trees.

\begin{thm}\label{lemma39} The following are equivalent:
\begin{itemize}
\item[(i)] There exists a special $\kappa$-Aronszajn tree;
\item[(ii)] There exists a streamlined $\kappa$-Aronszajn tree $K$, a club $D\s\acc(\kappa)$ and a function $f:K\restriction D\rightarrow\kappa$ such that all of the following hold:
\begin{itemize}
\item $V^-(K)\supseteq D$;
\item  $f(x)<\dom(x)$ for all $x\in K\restriction D$;
\item $f(x)\neq f(y)$ for every pair $x\subsetneq y$ of nodes from $K\restriction D$;
\item for all $x,y\in K$ and $\varepsilon\in\dom(x)\cap\dom(y)$, if $x(\varepsilon)=y(\varepsilon)$, then $x\restriction\varepsilon=y\restriction\varepsilon$.
\end{itemize}
\item[(iii)] There exists a streamlined uniformly homogeneous special $\kappa$-Aronszajn tree $T$ for which $V(T)$ covers a club in $\kappa$;
\item[(iv)] There exists an homogeneous special $\kappa$-Aronszajn tree $\mathbf T$ with $V(\mathbf T)=\acc(\kappa)$.
\end{itemize}
\end{thm}
\begin{proof} $(i)\implies(ii)$ Assuming that there exists a special $\kappa$-Aronszajn tree,
by \cite[Lemma~1.2 and Theorem~2.5]{MR3078820}, we may fix a $C$-sequence $\vec C=\langle C_\beta\mid\beta<\kappa\rangle$ and a club $C\s\acc(\kappa)$ satisfying the following:
\begin{enumerate}
\item for every $\beta\in C$, $\min(C_\beta)>\otp(C_\beta)$;
\item for every $\beta\in\acc(\kappa)\setminus C$, $\min(C_\beta)>\sup(C\cap\beta)$;
\item for every $\epsilon<\kappa$, $|\{ C_\beta\cap\epsilon\mid \beta<\kappa\}|<\kappa$.
\end{enumerate}
Consider the following additional requirement:
\begin{itemize}
\item[(4)] $\min(C_\beta)=\otp(C_\beta)+1$ for every $\beta\in C$.
\end{itemize}

\begin{claim}  We may moreover assume that Clause~(4) holds.
\end{claim}
\begin{proof} For every $\beta\in C$, let $C_\beta^\bullet:=C_\beta\cup\{\otp(C_\beta)+1\}$,
and for every $\beta\in\kappa\setminus C$, let $C_\beta^\bullet:=C_\beta$.
We just need to verify that $|\{ C^\bullet_\beta\cap\epsilon\mid \beta<\kappa\}|<\kappa$
for every $\epsilon<\kappa$. Towards a contradiction, suppose that $\epsilon$ is a counterexample.
From $(3)$, it follows that we may fix $B\in[C]^\kappa$ on which the map $\beta\mapsto C_\beta^\bullet\cap\epsilon$ is injective.
We may moreover assume that $\beta\mapsto C_\beta\cap\epsilon$ is constant over $B$.
By possibly removing one element of $B$, we may assume that $C_\beta^\bullet\cap\epsilon$ is nonempty for all $\beta\in B$.
So, we may moreover assume the existence of $\tau<\epsilon$ such that $\min(C^\bullet_\beta)=\tau$ for every $\beta\in B$.
But then $C_\beta^\bullet\cap\epsilon=(C_\beta\cap\epsilon)\cup\{\tau\}$
for every $\beta\in B$. This is a contradiction.
\end{proof}

Now, let $\rho_0$ be the characteristic function from
\cite[\S6]{MR2355670} obtained by walking along $\vec C$ satisfying (1)--(4),
and consider the following streamlined $\kappa$-tree $$T(\rho_0):=\{ \rho_{0\beta}\restriction\alpha\mid \alpha\le\beta<\kappa\}.$$
Using (1)--(3), the proof of \cite[Theorem~4.4]{MR3078820} provides a club $D\s C$ and a function $g:T(\rho_0)\restriction D\rightarrow\kappa$ satisfying the following two:
\begin{itemize}
\item $g(t)<\dom(t)$ for all $t\in T(\rho_0)\restriction D$;
\item for every pair $s\subsetneq t$ of nodes from $T(\rho_0)\restriction D$, $g(s)\neq g(t)$.
\end{itemize}

Next, consider the following subfamily of $T(\rho_0)$:
$$T:=\{ \rho_{0\beta}\restriction\alpha\mid \alpha<\beta<\kappa\}.$$

Clearly, $T$ is downward-closed and $\{\dom(y)\mid y\in T\}=\kappa$,
so that $T$ is a streamlined $\kappa$-Aronszajn subtree of $T(\rho_0)$.

\begin{claim} $T\cap\{\rho_{0\alpha}\mid \alpha\in C\}=\emptyset$. In particular, $V^-(T)\supseteq C\supseteq D$.
\end{claim}
\begin{proof} The ``in particular'' part will follow from the fact that $\{ \rho_{0\alpha}\restriction\epsilon\mid \epsilon<\alpha\}$ is an $\alpha$-branch of $T$
for every $\alpha<\kappa$.
Thus, let $\alpha\in C$ and we shall prove that $\rho_{0\alpha}\notin T$.
Suppose not, and pick some $\beta>\alpha$ such that $\rho_{0\alpha}=\rho_{0\beta}\restriction\alpha$.
Recall that for every $\gamma<\kappa$, $$C_\gamma=\{ \xi<\gamma\mid \rho_{0\gamma}(\xi)\text{ is a sequence of length }1\}.$$
In particular, $\min(C_\alpha)=\min(C_\beta)$. As $\sup(C\cap\beta)\ge\alpha>\min(C_\alpha)$, it follows from Clause~(2) that $\beta\in C$.
So, by Clause~(4), $\otp(C_\alpha)=\otp(C_\beta)$. It follows that may fix some $\delta\in C_\alpha\setminus C_\beta$.
But then $\rho_{0\alpha}(\delta)$ is a sequence of length $1$, whereas $\rho_{0\beta}(\delta)$ is a longer sequence. This is a contradiction.
\end{proof}

For every $t\in T\restriction\acc(\kappa)$, define a function $k_t:\dom(t)\rightarrow T$ via
$$k_t(\varepsilon):=t\restriction\varepsilon.$$

Let $K$ be the following downward-closed subfamily of ${}^{<\kappa}H_\kappa$:
$$K:=\{ k_t\restriction\alpha\mid \alpha\le\dom(t), t\in T\restriction\acc(\kappa)\}.$$

Evidently, for all $x,y\in K$ and $\varepsilon\in\dom(x)\cap\dom(y)$, if $x(\varepsilon)=y(\varepsilon)$, then $x\restriction\varepsilon=y\restriction\varepsilon$.
In addition, $t\mapsto k_t$ constitutes an isomorphism between $(T\restriction\acc(\kappa),{\s})$ and $(K\restriction\acc(\kappa),{\s})$,
and hence $K$ is a streamlined $\kappa$-Aronszajn tree with $V^-(K)\supseteq D$.
The fact that the above map is an isomorphism also implies that a function $f:K\restriction D\rightarrow\kappa$ defined via $f(k_t):=g(t)$
satisfies that $f(x)<\dom(x)$ for all $x\in K\restriction D$,
and that  $f(x)\neq f(y)$ for every pair $x\subsetneq y$ of nodes from $K\restriction D$.

$(ii)\implies(iii)$:
Suppose that $K$ and $f:K\restriction D\rightarrow\kappa$ are as in Clause~(ii).
By possibly shrinking $D$, we may assume that for all $\beta\in D$ and $\alpha<\beta$, it is the case that $\omega\cdot\alpha<\beta$.

The operation of Definition~\ref{regdef} is associative,
so we may define a family $T$ to be the collection of all elements of the form $x_0*\cdots*x_n$ where\footnote{To clarify, in the special case that $n=0$, $x_0*\cdots*x_n$ stands for $x_0$.}
\begin{itemize}
\item[(a)] $n<\omega$,
\item[(b)] $x_i\in K$ for all $i\le n$, and
\item[(c)] $\dom(x_i)<\dom(x_{i+1})$ for all $i<n$.
\end{itemize}

It is clear that $t\restriction\alpha\in T$ for all $t\in T$ and $\alpha<\kappa$.
Thus, recalling the proof of Claim~\ref{c2171},
to establish that $T$ is a uniformly homogeneous streamlined $\kappa$-tree, it suffices to prove the following claim.
\begin{claim} $T_0=\{\emptyset\}$ and $T_\alpha=\{ x*y\mid x\in T\restriction\alpha, y\in K_\alpha\}$
for every nonzero $\alpha<\kappa$.
\end{claim}
\begin{proof} Suppose that $\alpha$ is a nonzero ordinal such that $T_\epsilon=\{ x*y\mid x\in T\restriction\alpha, y\in K_\epsilon\}$
for every $\epsilon<\alpha$. Let $t\in T_\alpha$. Pick a sequence $(x_0,\ldots,x_n)$ satisfying (a)--(c)
for which $t=x_0*\cdots*x_n$.

$\br$ If $n=0$, then $t=\emptyset*x_0$ with $\emptyset\in T\restriction\alpha$ and $x_0\in K_\alpha$.

$\br$ If $n=m+1$ for some $m<\omega$, then $t=x*y$ with $x:=x_0*\cdots*x_m$ in $T\restriction\alpha$ and $y:=x_{m+1}$ in $K_\alpha$.
\end{proof}

For each node $t\in T$, we define $n(t)$ and $x(t)$ by first letting $n(t)$ denote the least $n$ for which there exists a sequence $(x_0,\ldots,x_n)$ satisfying (a)--(c)
for which $t=x_0*\cdots*x_n$,
and then letting $x(t)$ be such an $x_n$. Note that $\dom(x(t))=\dom(t)$, and that  $K=\{ t\in T\mid n(t)=0\}$.

Define a function $g:T\restriction D\rightarrow\kappa$ via
$$g(t):=(\omega\cdot f(x(t)))+n(t).$$
\begin{claim} \begin{enumerate}
\item $g(t)<\dom(t)$ for all $t\in T\restriction D$;
\item Let $s\subsetneq t$ be a pair of nodes from $T\restriction D$. Then $g(s)\neq g(t)$.
\end{enumerate}
\end{claim}
\begin{proof} (1) Since $\omega\cdot\alpha<\beta$ for all $\beta\in D$ and $\alpha<\beta$.

(2) Suppose not. Let $\tau<\kappa$ and $n<\omega$ be such that
$f(x(s))=\tau=f(x(t))$ and $n(s)=n=n(t)$. By the choice of $f$ it follows that $x(s)\nsubseteq x(t)$,
so since $s\subsetneq t$, it must be the case that $n=m+1$ for some $m<\omega$.
Fix a sequence $(x_0,\ldots,x_m,x_{m+1})$ of nodes from $K$ such that $s=x_0*\cdots * x_m* x_{m+1}$ and $x_{m+1}=x(s)$.
Likewise, fix a sequence $(y_0,\ldots,y_m,y_{m+1})$ of nodes from $K$ such that $t=y_0*\cdots * y_m* y_{m+1}$ and $y_{m+1}=x(t)$.

$\br$ As $x_{m+1}\nsubseteq y_{m+1}$, we may fix $\delta\in\dom(x_{m+1})$ such that $x_{m+1}(\delta)\neq y_{m+1}(\delta)$.

$\br$ As $s\s t= y_0*\cdots * y_m* y_{m+1}$ and $n(s)>m$, it must be the case that $\dom(y_m)<\dom(s)$.

Altogether, $\varepsilon:=\max\{\delta+1,\dom(x_m),\dom(y_m)\}$ is an ordinal less than $\dom(s)$,
satisfying $x_{m+1}(\varepsilon)=s(\varepsilon)=t(\varepsilon)=y_{m+1}(\varepsilon)$, but then $x_{m+1}\restriction\varepsilon=y_{m+1}\restriction\varepsilon$, contradicting the fact that $\delta<\varepsilon$.
\end{proof}

It is easy to see that the two features of $g$ together imply that $T$ admits no $\kappa$-branch.
The beginning of the proof of \cite[Theorem~4.4]{MR3078820} shows furthermore that $T$ must be a special $\kappa$-Aronszajn tree.

\begin{claim} $V(T)\supseteq D$.
\end{claim}
\begin{proof} Let $\alpha\in D$. As $D\s V^-(K)$, we may fix a function $t:\alpha\rightarrow H_\kappa$ such that $\{ t\restriction\epsilon\mid \epsilon<\alpha\}\s K$,
but $t\notin K$. As $K\s T$, it thus suffices to prove that $t\notin T$.
Towards a contradiction, suppose that $t\in T$. In particular, $n(t)>0$.
Fix $m<\omega$ and a sequence $(x_0,\ldots,x_m,x_{m+1})$ of nodes from $K$ such that $t=x_0*\cdots * x_m*x_{m+1}$.
As $x_{m+1}\neq t$, we may fix some $\delta<\alpha$ such that $t(\delta)\neq x_{m+1}(\delta)$.
Pick $\varepsilon<\alpha$ above $\max\{\delta,\dom(x_m)\}$.
Then $t(\varepsilon)=x_{m+1}(\varepsilon)$. But $t\restriction(\varepsilon+1)$ and $x_{m+1}\restriction(\varepsilon+1)$ are two nodes in $K$
that agree on $\varepsilon$ and hence $t\restriction(\varepsilon+1)=x_{m+1}\restriction(\varepsilon+1)$, contradicting the fact that $\delta<\varepsilon$.
\end{proof}

The implication $(iii)\implies(iv)$ follows from the proof of Lemma~\ref{lemma33} and the implication $(iv)\implies(i)$ is trivial.
\end{proof}

\begin{defn}[Products]\label{def-classical-product-tree}
For a sequence of $\kappa$-trees $\langle \mathbf{T}^i \mid i<\tau \rangle$
with $\mathbf T^i = (T^i, {<_{T^i}})$ for each $i<\tau$,
the product
$\bigotimes_{i<\tau} \mathbf{T}^i$
is defined to be the tree $\mathbf T=({T}, {<_{{T}}})$,
where:
\begin{itemize}
\item $T=\bigcup\{\prod_{i<\tau}T_\alpha^i\mid \alpha<\kappa\}$;
\item $\vec{s} <_{{T}} \vec{t}$ iff $\vec s(i) <_{T^i} \vec t(i)$ for every $i<\tau$.
\end{itemize}
\end{defn}
\begin{prop}\label{products}

For a sequence $\langle \mathbf{T}^i \mid i<\tau \rangle$ of normal $\kappa$-trees,
if $\lambda^\tau<\kappa$ for all $\lambda<\kappa$, then:
\begin{enumerate}
\item $\bigotimes_{i<\tau} \mathbf{T}^i$ is a normal $\kappa$-tree;
\item $V(\bigotimes_{i<\tau} \mathbf{T}^i)=\bigcup\{V(\mathbf T^i)\mid i<\tau\}$;
\item $V^-(\bigotimes_{i<\tau} \mathbf{T}^i)=\bigcup\{V^-(\mathbf T^i)\mid i<\tau\}$.
\end{enumerate}
\end{prop}
\begin{proof} Left to the reader.
\end{proof}

\begin{defn}[Sums] The \emph{disjoint sum} $\sum\mathcal P$ of a family of posets $\mathcal P$
is the poset $(A,<_A)$ defined as follows:
\begin{itemize}
\item $A:=\{ ((P,<_P),x)\mid (P,<_P)\in\mathcal P, x\in P\}$;
\item $((P,<_P),x)<_A ((Q,<_Q),y)$ iff $(P,<_P)=(Q,<_Q)$ and $x<_Py$.
\end{itemize}
\end{defn}
In the special case of doubleton we write $\mathbf T+\mathbf S$ instead of $\sum\{\mathbf T,\mathbf S\}$.
\begin{prop}\label{sums} Suppose that $\mathcal T$ is a family of less than $\kappa$ many $\kappa$-trees. Then:
\begin{enumerate}
\item $\sum\mathcal T$ is a $\kappa$-tree;
\item $V(\sum\mathcal T)=\bigcap\{V(\mathbf T)\mid \mathbf T\in \mathcal T\}$;
\item $V^-(\sum\mathcal T)=\bigcup \{V^-(\mathbf T)\mid \mathbf T\in \mathcal T\}$.
\end{enumerate}
\end{prop}
\begin{proof} Left to the reader.
\end{proof}

It follows from Propositions \ref{products} and \ref{sums} that
$\vspec(\kappa)$ is closed under finite unions and intersections.

\begin{cor}\label{cor224} Suppose $\chi\in\reg(\kappa)$ is such that $\lambda^{<\chi}<\kappa$ for all $\lambda<\kappa$.
Then there exists a $\kappa$-tree $\mathbf T$ with $V^-(\mathbf T)\supseteq \acc(\kappa)\cap E^\kappa_{\leq \chi }$.
\end{cor}
\begin{proof} Denote $\Theta:=\reg(\chi+1)$.
By Lemmas \ref{l22} and \ref{lemma311},
for every  $\theta\in\Theta$, we may pick a $\kappa$-tree  $\mathbf T^\theta$ such that $V^-(\mathbf T^\theta)$ covers $E^\kappa_{\theta}$
modulo a club. In fact, the proof of $(2)\implies(3)$ of Lemma~\ref{lemma311} shows that we may secure $V^-(\mathbf T^\theta)\supseteq E^\kappa_{\theta}$.
Let $\mathbf T:=\sum\{\mathbf T^\theta\mid \theta\in\Theta\}$ be the disjoint sum of these trees.
By Proposition~\ref{sums}, $V^-(\mathbf T)=\bigcup_{\theta\in\Theta}V^-(\mathbf T^\theta)\supseteq\bigcup_{\theta\in\Theta}E^\kappa_\theta= \acc(\kappa)\cap E^\kappa_{\leq \chi }$.
\end{proof}
\begin{remark} In Section~\ref{sect6}, we provide sufficient conditions for getting an homogeneous $\kappa$-Souslin tree $\mathbf T$ with $V(\mathbf T)=\bigcup_{\chi\in x}E^\kappa_\chi$ for a prescribed finite and nonempty $x\s\reg(\kappa)$.
\end{remark}

\begin{question} Is it consistent that for some regular uncountable cardinal $\kappa$, there are $\kappa$-Souslin trees, but $V(\mathbf T)$ is nonstationary for every $\kappa$-Souslin tree $\mathbf T$?
\end{question}

By Proposition~\ref{prop22},
Corollary~\ref{cor26} and \cite[Lemma~2.4]{paper20},
in such a model there cannot be an homogeneous $\kappa$-Souslin tree.
A model with an $\aleph_1$-Souslin tree but no homogeneous one was constructed by Abraham and Shelah in \cite{Sh:403}.

\section{Consulting another tree}\label{sec4}

The main result of this section is Theorem~\ref{thm41} below.
A sample corollary of it reads as follows.

\begin{cor}\label{cor31} Suppose that $\kappa=\lambda^+$ for an infinite cardinal $\lambda$.
\begin{enumerate}
\item If $\square_\lambda\mathrel{+}\diamondsuit(\kappa)$ holds,
then there exists a $\kappa$-Souslin tree $\mathbf T$ with $V(\mathbf T)=\acc(\kappa)$;
\item If $\square(\kappa)$ holds and $\aleph_0<\lambda^{<\lambda}<\lambda^+=2^\lambda$,
then there exists a $\kappa$-Souslin tree $\mathbf T$ with $V(\mathbf T)=\acc(\kappa)$;
\item If $\p_\lambda(\kappa,\kappa,{\sq},1)$ holds, then there exists a $\kappa$-Souslin tree $\mathbf T$ such that $V(\mathbf T)\supseteq E^\kappa_{>\omega}$.
\end{enumerate}
\end{cor}
\begin{proof} (1) $\diamondsuit(\aleph_1)$ implies the existence of a normal and splitting $\aleph_1$-Souslin tree $\mathbf T$,
and by Corollary~\ref{cor27}, $V(\mathbf T)=\acc(\aleph_1)$.
For $\lambda\ge\aleph_1$, by \cite[Corollary~3.9]{paper22}, $\square_\lambda+\ch_\lambda$ is equivalent to $\p_\lambda(\kappa,2,{\sq},1)$.
In addition, by a theorem of Jensen, $\square_\lambda$ gives rise to a special $\lambda^+$-Aronszajn tree. Thus, we infer from Proposition~\ref{lemma39}
the existence of a $\kappa$-tree $\mathbf K$ for which $V^-(\mathbf K)=\acc(\kappa)$.
It thus follows from Theorem~\ref{thm41}(1) below that there exists a $\kappa$-Souslin tree $\mathbf T$ for which $V(\mathbf T)$ is a club in $\kappa$.
Finally, appeal to Lemma~\ref{lemma33}.

(2) By \cite[Corollary~4.4]{paper24}, the hypothesis implies that $\p^-(\kappa,2,{\sq},1)$ holds.
In addition, by a theorem of Specker, $\lambda=\lambda^{<\lambda}$ implies the existence of a special $\lambda^+$-Aronszajn tree.
Now, continue as in the proof of Clause~(1).

(3) Similar to the proof of Clause~(1),
using Theorem~\ref{thm41}(2), instead.
\end{proof}
\begin{remark} Sufficient conditions for $\p_\lambda(\kappa,\kappa,{\sq},1)$ to hold are given by Corollaries 3.15 and 3.24 of \cite{paper32}.
\end{remark}

Before turning to the proofs of the main results of this section, we provide a few preliminaries.

\begin{defn}[Proxy principle, \cite{paper22,paper23}]\label{proxydef}
Suppose that $\mu,\theta\le\kappa$ are cardinals, $\xi\le\kappa$ is an ordinal,
$\mathcal R$ is a binary relation over $[\kappa]^{<\kappa}$ and $\mathcal S$ is a collection of stationary subsets of $\kappa$.
The principle $\p_\xi^-(\kappa,\mu,\mathcal{R},\theta,\mathcal{S})$ asserts the existence of a $\xi$-bounded $\mathcal C$-sequence $\langle \mathcal{C}_\alpha\mid \alpha<\kappa\rangle$ such that:
\begin{itemize}
\item for every $\alpha<\kappa$, $|\mathcal C_\alpha|<\mu$;
\item for all $\alpha<\kappa$, $C\in \mathcal{C}_\alpha$, and $\bar{\alpha}\in\acc(C)$, there exists some $D\in \mathcal{C}_{\bar{\alpha}}$ such that $D\mathrel{\mathcal{R}}C$;
\item for every sequence $\langle B_i\mid i<\theta\rangle$ of cofinal subsets of $\kappa$, and every $S\in\mathcal{S}$, there are stationarily many $\alpha\in S$
such that for all $C\in\mathcal C_\alpha$ and $i<\min\{\alpha,\theta\}$, $\sup(\nacc(C)\cap B_i)=\alpha$.
\end{itemize}
\end{defn}
\begin{conv}\label{proxydef2} We write $\p_\xi(\kappa, \mu, \mathcal R, \theta, \mathcal S)$
to assert that $\p^-_\xi(\kappa, \mu, \mathcal R, \theta, \mathcal S)$
and $\diamondsuit(\kappa)$ both hold.
\end{conv}
\begin{conv}\label{conv35} If we omit $\xi$, then we mean $\xi:=\kappa$.
If we omit $\mathcal S$, then we mean $\mathcal S:=\{\kappa\}$.
In the case $\mu=2$, we identify $\langle \mathcal{C}_\alpha\mid \alpha<\kappa\rangle$ with the unique element $\langle C_\alpha \mid \alpha < \kappa\rangle$
of $\prod_{\alpha<\kappa}\mathcal{C}_\alpha$.
\end{conv}

\begin{fact}[{\cite[Lemma~2.2]{paper22}}]\label{Diamond_H_kappa} The following are equivalent:
\begin{enumerate}
\item $\diamondsuit(\kappa)$, i.e., there is a sequence $\langle f_\beta\mid \beta<\kappa\rangle$ such that
for every function $f:\kappa\rightarrow\kappa$, the set $\{ \beta<\kappa\mid f\restriction\beta=f_\beta\}$ is stationary in $\kappa$.
\item $\diamondsuit^-(H_\kappa)$, i.e., there is a sequence $\langle \Omega_\beta \mid \beta < \kappa \rangle$ such that for all $p\in H_{\kappa^{+}}$  and $\Omega \subseteq H_\kappa$,
there exists an elementary submodel $\mathcal M\prec H_{\kappa^{+}}$ such that:
\begin{itemize}
\item $p\in\mathcal M$;
\item $\mathcal M\cap\kappa\in\kappa$;
\item $\mathcal M\cap \Omega=\Omega_{\mathcal M\cap\kappa}$.
\end{itemize}
\item $\diamondsuit(H_\kappa)$, i.e., there are a partition $\langle R_i \mid  i < \kappa \rangle$ of $\kappa$
and a sequence $\langle \Omega_\beta \mid \beta < \kappa \rangle$ such that for all $p\in H_{\kappa^{+}}$, $\Omega \subseteq H_\kappa$,
and $i<\kappa$,
there exists an elementary submodel $\mathcal M\prec H_{\kappa^{+}}$ such that:
\begin{itemize}
\item $p\in\mathcal M$;
\item $\mathcal M\cap\kappa\in  R_i$;
\item $\mathcal M\cap \Omega=\Omega_{\mathcal M\cap\kappa}$.
\end{itemize}
\end{enumerate}
\end{fact}

\begin{thm}\label{thm41}
Suppose that $K$ is some streamlined $\kappa$-tree.

\begin{enumerate}
\item 	If $\p(\kappa,2,{\sq^*},1)$ holds, then there exists a normal and splitting streamlined $\kappa$-Souslin tree $T$ such that $V(T)\supseteq V^-(K)$;
\item 	If $\p(\kappa,\kappa,{\sq},1)$ holds, then there exists a normal and splitting streamlined $\kappa$-Souslin tree $T$ such that $V(T)\supseteq V^-(K)\cap E^\kappa_{>\omega}$.
\end{enumerate}
\end{thm}
\begin{proof} Fix a well-ordering $\vartriangleleft$ of $H_\kappa$,
and a sequence $\langle \Omega_\beta\mid \beta<\kappa\rangle$ witnessing $\diamondsuit^-(H_\kappa)$.
If $\p^-(\kappa,\kappa,{\sq},1)$ holds,
then let $\vec {\mathcal C}=\langle \mathcal C_\alpha\mid \alpha<\kappa\rangle$ be any $\p^-(\kappa,\kappa,{\sq},1)$-sequence.
If $\p^-(\kappa,2,{\sq^*},1)$ holds,
then, by \cite[Theorem~4.39]{paper23}, we may let $\vec {\mathcal C}=\langle \mathcal C_\alpha\mid \alpha<\kappa\rangle$ be a $\p^-(\kappa,\kappa,{\sq},1)$-sequence
with the added feature that for every $\alpha\in\acc(\kappa)$ for all $C,D\in\mathcal C_\alpha$, $\sup(C\symdiff D)<\alpha$.

Following the proof of \cite[Proposition~2.2]{paper26}, we shall recursively construct a sequence $\langle T_\alpha\mid \alpha<\kappa\rangle$ such that
$T:=\bigcup_{\alpha<\kappa}T_\alpha$ will constitute the tree of interest whose $\alpha^{\text{th}}$-level is $T_\alpha$.

We start by letting $T_0:=\{\emptyset\}$,
and once $T_\alpha$ has already been defined, we let $$T_{\alpha+1}:=\{t{}^\smallfrown\langle 0\rangle,t{}^\smallfrown\langle 1\rangle,t{}^\smallfrown\langle \eta\rangle\mid t\in T_\alpha, \eta\in K_\alpha\}.$$

Next, suppose that $\alpha\in \acc(\kappa)$ is such that $T\restriction\alpha$ has already been defined.
For all $C\in\mathcal C_\alpha$ and $x\in T\restriction C$,
we shall identify a set of potential nodes $\{\mathbf b_x^{C,\eta}\mid \eta\in\mathcal B(K\restriction\alpha)\}$
and then let
\begin{equation}\tag{$\star$}\label{promise0}T_\alpha:=\{\mathbf b_x^{C,\eta}\mid C\in\mathcal C_\alpha,\eta\in K_\alpha, x\in T\restriction C\}.\end{equation}

To this end, fix $C\in\mathcal C_\alpha$, $x\in T\restriction C$ and $\eta\in\mathcal B(K\restriction\alpha)$.
The node $\mathbf b_x^{C,\eta}$ will be obtained as the limit $\bigcup\im(b_x^{C,\eta})$ of a sequence $b_x^{C,\eta}\in\prod_{\beta\in C\setminus\dom(x)}T_\beta$, as follows:
\begin{itemize}
\item	Let $b_x^{C,\eta}(\dom(x)):=x$.

\item For every $\beta\in \nacc(C)$ above $\dom(x)$ such that $b_x^{C,\eta}(\beta^-)$ has already been defined
for $\beta^-:=\sup(C\cap \beta)$, let
$$Q^{C, \eta}_x(\beta) := \{ t\in T_\beta\mid \exists s\in \Omega_{\beta}[ (s\cup(b_x^{C,\eta}(\beta^-){}^\smallfrown\langle \eta\restriction\beta^-\rangle))\s t]\}.$$
Now, consider the two possibilities:
\begin{itemize}
\item If $Q^{C,\eta}_x(\beta) \neq \emptyset$, then let $b^{C,\eta}_x(\beta)$ be its $\lhd$-least element;
\item Otherwise, let $b^{C,\eta}_x(\beta)$ be the $\lhd$-least element of $T_\beta$ that extends $b_x^{C,\eta}(\beta^-){}^\smallfrown\langle \eta\restriction\beta^-\rangle$.
Such an element must exist, as the level $T_\beta$ was constructed so as to preserve normality.
\end{itemize}

\item 	For every $\beta\in \acc(C\setminus \dom(x))$ such that $b_x^{C,\eta}\restriction\beta$ has already been defined,
let $b_x^{C,\eta}(\beta):=\bigcup \im(b_x^{C,\eta}\restriction\beta)$.
\end{itemize}
For the last case, we need to argue that $b_x^{C,\eta}(\beta)$ is indeed an element of $T_\beta$.
As $\vec{\mathcal C}$ is $\sq$-coherent, the set $\bar C:=C\cap \beta$ is in $\mathcal C_\beta$. Also, $K$ is a tree and hence $\bar\eta:=\eta\restriction\beta$ is in $K_\beta$.
So, since $\mathbf b_x^{\bar C,\eta\restriction\beta}\in T_\beta$, to show that $b_x^{C,\eta}(\beta)\in T_\beta$, it suffices to prove the following.

\begin{claim}\label{c371} $b_x^{C,\eta}(\beta)=\mathbf b_x^{{\bar C},\bar\eta}$.
\end{claim}
\begin{proof}
Clearly, $\dom(b_x^{C,\eta}(\beta))=C\cap \beta\setminus \dom(x)={\bar C}\setminus\dom(x)=\dom(b_x^{{\bar C},\bar\eta})$.
So, we are left with showing that $b_x^{C,\eta}(\delta)=b_x^{{\bar C},\bar\eta}(\delta)$ for all $\delta\in {\bar C}\setminus\dom(x)$.
The proof is by induction on $\delta\in {\bar C}\setminus\dom(x)$:
\begin{itemize}
\item For $\delta=\dom(x)$, we have that $b_x^{\eta,C}(\delta)=x=b_x^{{\bar C},\bar\eta}(\delta)$.
\item Given $\delta\in \nacc({\bar C})$ above $\dom(x)$ such that $b_x^{C,\eta}(\delta^-)=b_x^{{\bar C},\bar\eta}(\delta^-)$ for $\delta^-:=\sup({\bar C}\cap\delta)$, we argue as follows.
Since
$$b_x^{C,\eta}(\delta^-){}^\smallfrown\langle \eta\restriction\delta^-\rangle=b_x^{{\bar C},\bar \eta}(\delta^-){}^\smallfrown\langle \bar\eta\restriction\delta^-\rangle,$$
the definitions of $b_x^{C,\eta}(\delta)$ and $b_x^{{\bar C},\bar\eta}(\delta)$ coincide.
\item If $\delta\in\acc({\bar C}\setminus\dom(x))$, then we take the limit of two identical sequences, and the unique limit is identical. 		\qedhere
\end{itemize}
\end{proof}

This completes the definition of $b_x^{C,\eta}$. For all $\eta\in\mathcal B(K\restriction\alpha)$, let $\mathbf b_x^{C,\eta}:=\bigcup\im(b_x^{C,\eta})$, and then
we define $T_\alpha$ as promised in \eqref{promise0}.

Clearly, $T:=\bigcup_{\alpha<\kappa}T_\alpha$ is a normal and splitting $\kappa$-tree.
The verification of Souslin-ness is standard (see \cite[Claims 2.2.2 and 2.2.3]{paper26}).
\begin{claim}\label{c372} Suppose that $\alpha\in V^-(K)$ is such that $\sup(C\cap D)=\alpha$
for all $C,D\in\mathcal C_\alpha$. Then $\alpha\in V(T)$.
\end{claim}
\begin{proof}
As $\alpha\in V^-(K)$, we may fix $\eta\in\mathcal B(K\restriction\alpha)\setminus K_\alpha$.
Let $x\in T\restriction\alpha$, and we shall find a vanishing $\alpha$-branch through $x$ in $T$.
First fix $C\in\mathcal C_\alpha$.
Using normality and by possibly extending $x$, we may assume that $x\in T\restriction C$.
We have already established that $\{ \mathbf b_x^{C,\eta}\restriction\epsilon\mid \epsilon<\alpha\}$  is an $\alpha$-branch through $x$.
Towards a contradiction, suppose that it is not vanishing, so that $\bigcup\im(b_x^{C,\eta})$ is in $T_\alpha$.
It follows from \eqref{promise0} that we may pick $D\in\mathcal C_\alpha$, $y\in T\restriction D$ and $\xi\in K_\alpha$ such that $\bigcup\im(b_x^{C,\eta})=\mathbf b_y^{D,\xi}$.
Fix $\beta\in C\cap D$ large enough such that $\beta>\max\{\dom(x),\dom(y)\}$
and $\eta\restriction\beta\neq\xi\restriction\beta$.
In particular, $\beta\in \dom(b_x^{C,\eta})\cap \dom(b_y^{D,\xi})$.
Consider $\beta^C:=\min(C\setminus \beta+1)$, the successor of $\beta$ in $C$
and $\beta^D:=\min(D\setminus \beta+1)$, the successor of $\beta$ in $D$.
Then the definition of the successor stage of $b_x^{C,\eta}$ ensures that $b_x^{C,\eta}(\beta^C)$ extends
$b_x^{C,\eta}(\beta){}^\smallfrown\langle \eta\restriction\beta\rangle$,
so that 	$b_x^{C,\eta}(\beta^C)(\beta)=\eta\restriction\beta$.
Likewise, 	$b_y^{D,\xi}(\beta^D)(\beta)=\xi\restriction\beta$.
From 	$\mathbf b_x^{C,\eta}=\mathbf b_y^{D,\xi}$,
we infer that $b_x^{C,\eta}(\beta^C)(\beta)=\mathbf b_x^{C,\eta}(\beta)=\mathbf b_y^{D,\xi}(\beta)=b_y^{D,\xi}(\beta^D)(\beta)$,
contradicting the fact that $\eta\restriction\beta\neq\xi\restriction\beta$.
\end{proof}
This completes the proof.
\end{proof}

We now arrive at Theorem~\ref{thmc}:
\begin{cor}\label{cor310} Suppose that $\p(\kappa,2,{\sq^*},1)$ holds.
Then:
\begin{enumerate}
\item For every $\chi\in\reg(\kappa)$ such that $\lambda^{<\chi}<\kappa$ for all $\lambda<\kappa$,
and every $\kappa$-tree $\mathbf K$,
there exists a $\kappa$-Sousin tree $\mathbf T$ such that $(E^\kappa_{\le\chi}\cup V^-(\mathbf K))\setminus V(\mathbf T)$ is nonstationary;
\item There exists a $\kappa$-Sousin tree $\mathbf T$ such that $V(\mathbf T)$ is stationary.
\end{enumerate}
\end{cor}
\begin{proof} (1) Suppose $\chi$ and $\mathbf K$ are as above.
By Corollary~\ref{cor224}, we may fix a $\kappa$-tree $\mathbf H$ with $V^-(\mathbf H)\supseteq \acc(\kappa)\cap E^\kappa_{\leq \chi }$.
By Proposition~\ref{sums}, $\mathbf K+\mathbf H$ is a $\kappa$-tree with $V^-(\mathbf K+\mathbf H)=V^-(\mathbf K)\cup V^-(\mathbf H)$.
By \cite[Lemma~2.5]{paper23}, we may fix a streamlined $\kappa$-tree that $K$ that is club-isomorphic to $\mathbf K+\mathbf H$.
Now, appeal to Theorem~\ref{thm41}(1) with $K$.

(2) Appeal to Clause~(1) with $\chi=\omega$.
\end{proof}

\begin{defn}[Jensen-Kunen, \cite{jensen1969some}] A cardinal $\kappa$ is \emph{subtle} iff for every list $\langle A_\alpha\mid\alpha\in D\rangle$ over a club $D\s\kappa$,
there is a pair $(\alpha,\beta)\in[D]^2$ such that $A_\alpha=A_\beta\cap\alpha$.
\end{defn}

We now arrive at Theorem~\ref{thmb}:

\begin{cor}
We have $(1)\implies(2)\implies(3)\implies(4)$:
\begin{enumerate}
\item there exists a $\kappa$-Souslin tree $\mathbf T$ such that $V(\mathbf T)=\acc(\kappa)$;
\item there exists a $\kappa$-tree $\mathbf T$ such that $V(\mathbf T)=\acc(\kappa)$;
\item there exists a $\kappa$-tree $\mathbf T$ such that $V^-(\mathbf T)$ contains a club in $\kappa$;
\item $\kappa$ is not subtle.
\end{enumerate}

In addition, in $\mathsf{L}$, for $\kappa$ not weakly compact, $(4)\implies(1)$.
\end{cor}
\begin{proof} $(1)\implies(2)\implies(3)$: This is immediate.

$(3)\implies(4)$: By Lemma~\ref{lemma311}.

Next, work in $\mathsf L$ and suppose that $\kappa$ is a regular uncountable cardinal that is not subtle and not weakly compact.
If $\kappa$ is a successor cardinal, then by Corollary~\ref{cor31}(1), Clause~(1) holds,
so assume that $\kappa$ is inaccessible. By $\gch$, $\kappa$ is moreover strongly inaccessible,
and then Lemma~\ref{lemma311} yields that Clause~(3) holds.
Since we work in $\mathsf{L}$ and $\kappa$ is not weakly compact, by \cite[Theorem~3.12]{paper22},
$\p(\kappa,2,{\sq},1)$ holds.
So by Corollary~\ref{cor310}(1), Clause~(3) yields a $\kappa$-Souslin tree $\mathbf T$ such that $V(\mathbf T)$ covers a club in $\kappa$.
Now, appeal to Lemma~\ref{lemma33}.
\end{proof}

\begin{cor} In $\mathsf{L}$, if $\kappa$ is not weakly compact,
then for every stationary $S\s\kappa$, there exists a $\kappa$-Souslin tree $\mathbf T$ for which $V(\mathbf T)\cap S$ is stationary.
\end{cor}
\begin{proof} By Corollary~\ref{cor31}(1), we may assume that $\kappa$ is (strongly) inaccessible.
By Corollary~\ref{cor213}, we may fix a $\kappa$-tree $\mathbf K$ such that $V^-(\mathbf K)\cap S$ is stationary.
By \cite[Theorem~3.12]{paper22}, $\p(\kappa,2,{\sq},1)$ holds.
Finally, appeal to Corollary~\ref{cor310}(1).
\end{proof}

\section{Realizing a nonreflecting stationary set}\label{sect5}
In this section, we provide conditions concerning a set $S\s\kappa$
sufficient to ensure the existence of a $\kappa$-Souslin tree $\mathbf T$ with $V(\mathbf T)\supseteq S$ and possibly $V(\mathbf T)=S$.
As a corollary, we obtain Theorem~\ref{thmd}:
\begin{cor}  If $\diamondsuit(S)$ holds for some nonreflecting stationary subset $S$ of a strongly inaccessible cardinal $\kappa$,
then there is an almost disjoint family $\mathcal S$ of $2^\kappa$ many stationary subsets of $S$ such that,
for every $S'\in\mathcal S$, there is a $\kappa$-Souslin tree $\mathbf T$ with $V^-(\mathbf T)=V(\mathbf T)=S'$.
\end{cor}
\begin{proof} By Corollary~\ref{cor55} below, it suffices to prove that
there exists a family $\mathcal S$ of $2^\kappa$ many stationary subsets of $S$ such that:
\begin{itemize}
\item for every $S'\in\mathcal S$, $\diamondsuit(S')$ holds.
\item $|S'\cap S''|<\kappa$ for all $S'\neq S''$ from $\mathcal S$.
\end{itemize}

Now, as $\diamondsuit(S)$ holds, we may easily fix a sequence $\langle (A_\beta,B_\beta)\mid \beta\in S\rangle$
such that, for all $A,B\in\mathcal P(\kappa)$, the following set is stationary
$$G_A(B):=\{\beta\in S\mid A\cap\beta=A_\beta\ \&\ B\cap\beta=B_\beta\}.$$
Set $\mathcal S:=\{ S_A\mid A\in\mathcal P(\kappa)\}$,
where $S_A:=\{\beta\in S\mid A\cap\beta=A_\beta\}$.
Then $\mathcal S$ is an almost disjoint family of $2^\kappa$ many stationary subsets of $S$,
and for every $S'\in\mathcal S$, $\diamondsuit(S')$ holds,
as witnessed by $\langle B_\beta\mid\beta\in S'\rangle$.
\end{proof}

\begin{defn}[\cite{paper22}]
A streamlined tree $T\s{}^{<\kappa}H_\kappa$ is \emph{prolific} iff for all $\alpha<\kappa$ and $t\in T_\alpha$,
$\{ t{}^\smallfrown\langle i\rangle\mid i<\max\{\omega,\alpha\}\}\s T$.
\end{defn}
A prolific tree is clearly splitting.

\begin{thm}\label{thm52}
Suppose that $\p(\kappa,\kappa,\sqleftup{S},1)$ holds for a given $S\s\acc(\kappa)$.
Then there exists a normal, prolific, streamlined $\kappa$-Souslin tree $T$ such that $ V(T)\supseteq S$.
\end{thm}
\begin{proof}
Fix a well-ordering $\vartriangleleft$ of $H_\kappa$,
a sequence $\langle \Omega_\beta\mid \beta<\kappa\rangle$ witnessing $\diamondsuit^-(H_\kappa)$,
and a sequence $\cvec C=\langle \mathcal C_\alpha\mid \alpha<\kappa\rangle$ witnessing $\p^-(\kappa,\kappa,{\sqleftup S},2)$.
By $\sqleftup S$-coherence, we may assume that for every $\alpha\in S$,
$\mathcal C_\alpha$ is a singleton.

Following the proof of \cite[Proposition~2.2]{paper26}, we shall recursively construct a sequence $\langle T_\alpha\mid \alpha<\kappa\rangle$ such that
$T:=\bigcup_{\alpha<\kappa}T_\alpha$ will constitute a normal prolific full streamlined $\kappa$-Souslin tree whose $\alpha^{\text{th}}$-level is $T_\alpha$.

Let $T_0:=\{\emptyset\}$, and for all $\alpha<\kappa$ let $$T_{\alpha+1}:=\{t{}^\smallfrown\langle i\rangle\mid t\in T_\alpha, i<\max\{\omega,\alpha\}\}.$$
Next, suppose that $\alpha\in\acc(\kappa)$ is such that $T\restriction\alpha$ has already been defined.
Constructing the level $T_\alpha$ involves deciding which branches through $T\restriction\alpha$ will have its limit placed into our tree.
For all $C\in\mathcal C_\alpha$ and $x\in T\restriction C$, we first define two $\alpha$-branches $\mathbf b_x^C$ and $\mathbf d_x^C$ such that $\{\mathbf b_x^C\mid x\in T\restriction C\}\cap\{\mathbf d_x^C\mid x\in T\restriction C\}=\emptyset$,
and then we  shall let:
\begin{equation}\tag{$\star$}\label{promise2}T_\alpha:=\begin{cases}\{\mathbf{b}_x^{C}\phantom{,\mathbf{d}_x^C}\mid C\in\mathcal C_\alpha, x\in T\restriction C\},&\text{if }\alpha\in S;\\
\{\mathbf{b}_x^C,\mathbf{d}_x^C\mid C\in\mathcal C_\alpha, x\in T\restriction C\},&\text{otherwise}.
\end{cases}\end{equation}
For every $\alpha\in S$, since $|\mathcal C|=1$, this ensures that $\alpha\in V(T)$.

Let $C\in\mathcal C$ and $x\in T\restriction C$.
We start by defining $\mathbf b_x^C$. It will be the limit $\bigcup\im(b_x^C)$ of a sequence $b_x^C\in\prod_{\beta\in C\setminus\dom(x)}T_\beta$ obtained by recursion, as follows.
Set $b_x^C(\dom(x)):=x$. At successor step, for every $\beta\in C\setminus(\dom(x)+1)$ such that $b_x^C(\beta^-)$ has already been defined with $\beta^-:=\sup(C\cap\beta)$,
we consult the following set:
$$Q^{C, \beta}_{x,0} := \{ t\in T_\beta\mid \exists s\in \Omega_{\beta}[ (s\cup (b^{C}_x(\beta^-){}^\smallfrown\langle0\rangle))\s t]\}.$$
Now, consider the two possibilities:
\begin{itemize}
\item If $Q^{C,\beta}_{x,0} \neq \emptyset$, then let $b^C_x(\beta)$ be its $\lhd$-least element;
\item Otherwise, let $b^C_x(\beta)$ be the $\lhd$-least element of $T_\beta$ that extends $b^C_x(\beta^-){}^\smallfrown\langle0\rangle$.
Such an element must exist, as the tree constructed so far is prolific and normal.
\end{itemize}
Finally, for every $\beta\in\acc(C\setminus\dom(x))$ such that $b_x^C\restriction\beta$ has already been defined, we let $b_x^C(\beta)=\bigcup\im(b_x^C\restriction\beta)$.
By \eqref{promise2}, $\sqleftup{S}$-coherence and the exact same proof of \cite[Claim~2.2.1]{paper26}, $b_x^C(\beta)$ is indeed in $T_\beta$.

Next, we define $\mathbf d_x^C$ as the limit of a sequence $d_x^C\in\prod_{\beta\in C\setminus\dom(x)}T_\beta$ obtained by recursion, as follows.
Set $d_x^C(\dom(x)):=x$. At successor step, for every $\beta\in C\setminus(\dom(x)+1)$ such that $d_x^C(\beta^-)$ has already been defined with $\beta^-:=\sup(C\cap\beta)$,
we consult the following set:
$$Q^{C, \beta}_{x,1} := \{ t\in T_\beta\mid \exists s\in \Omega_{\beta}[ (s\cup (d^{C}_x(\beta^-){}^\smallfrown\langle1\rangle))\s t]\}.$$
Now, consider the two possibilities:
\begin{itemize}
\item If $Q^{C,\beta}_{x,1} \neq \emptyset$, then let $d^C_x(\beta)$ be its $\lhd$-least element;
\item Otherwise, let $d^C_x(\beta)$ be the $\lhd$-least element of $T_\beta\setminus\{b_x^C(\beta)\}$ that extends $d^C_x(\beta^-){}^\smallfrown\langle1\rangle$.
Such an element must exist, as the tree constructed so far is prolific and normal.
\end{itemize}
Finally, for every $\beta\in\acc(C\setminus\dom(x))$ such that $d_x^C\restriction\beta$ has already been defined, we let $d_x^C(\beta)=\bigcup\im(d_x^C\restriction\beta)$.
By \eqref{promise2}, $\sqleftup{S}$-coherence and the exact same proof of \cite[Claim~2.2.1]{paper26}, $d_x^C(\beta)$ is indeed in $T_\beta$.

\begin{claim} For every $C\in\mathcal C_\alpha$,
$\{\mathbf b_x^C\mid x\in T\restriction C_\alpha\}\cap \{\mathbf d_x^C\mid x\in T\restriction C_\alpha\}=\emptyset$.
\end{claim}
\begin{proof} Let $C\in\mathcal C_\alpha$ and $x,y\in T\restriction C$.
Fix a large enough $\beta\in\nacc(C)$ for which $\beta^-:=\sup(C\cap\beta)$ is bigger than $\max\{\dom(x),\dom(y)\}$.
By the definitions of $b_x^C$ and $d_y^C$,
\begin{itemize}
\item $b_x^C(\beta)(\beta^-)=0$, and
\item $d_y^C(\beta)(\beta^-)=1$.
\end{itemize}
In particular, $\mathbf b_x^C\neq \mathbf d_y^C$.
\end{proof}

This finishes the construction of $T_\alpha$.
Finally, by \cite[Claims 2.2.2 and 2.2.3]{paper26},
$T:=\bigcup_{\alpha<\kappa}T_\alpha$ is a $\kappa$-Souslin tree.
\end{proof}

\begin{thm}\label{thm53}
Suppose that $\chi$ is a cardinal such that $\lambda^\chi<\kappa$ for all $\lambda<\kappa$,
and that $\p(\kappa,\kappa,\sqleftup S,1,\{S\cup E^\kappa_{>\chi}\})$ holds for a given $S\s \acc(\kappa)\cap E^\kappa_{\le\chi}$.
Then there exists a normal, prolific, streamlined $\kappa$-Souslin tree $T$ such that $V^-(T)\cap E^\kappa_{\le\chi}=V(T)\cap E^\kappa_{\le\chi}=S$.
\end{thm}
\begin{proof} The proof is almost identical to that of Theorem~\ref{thm52}, where the only change is in
that now, the definition of $T_\alpha$ for a limit $\alpha$ splits into three:
$$T_\alpha:=\begin{cases}\{\mathbf{b}_x^{C}\phantom{,\mathbf{d}_x^C}\mid C\in\mathcal C_\alpha, x\in T\restriction C\},&\text{if }\alpha\in S;\\
\{\mathbf{b}_x^C,\mathbf{d}_x^C\mid C\in\mathcal C_\alpha, x\in T\restriction C\},&\text{if }\alpha\in E^\kappa_{>\chi};\\
\mathcal B(T\restriction\alpha),&\text{otherwise}.\\
\end{cases}$$
The details are left to the reader.
\end{proof}
\begin{remark} Sufficient conditions for the existence of $S\s\kappa$
for which $\p(\kappa,\kappa,\sqleftup{S},1,\{S\})$ holds are given by \cite[Corollary~4.22]{paper23} and \cite[Theorem~4.28]{paper23}.
In particular, for every (nonreflecting) stationary $E\s\kappa$, if $\square(E)$ and $\diamondsuit(E)$ both hold,
then there exists a stationary $S\s E$ such that $\p(\kappa,\kappa,\sqleftup{S},1,\{S\})$ holds.
\end{remark}
\begin{cor}\label{cor46} Suppose that $2^{2^{\aleph_0}}=\aleph_2$, and that $S$ is a nonreflecting stationary subset of $E^{\aleph_2}_{\aleph_0}$.
Then there exists a normal prolific streamlined $\aleph_2$-Souslin tree $T$ such that $ V(T)=S\cup E^{\aleph_2}_{\aleph_1}$.
\end{cor}
\begin{proof} By \cite[Lemma~3.2]{paper32}, the hypotheses implies that $\p(\aleph_2,\aleph_2,{\sqleftup S},\allowbreak1,\{S\})$ holds.
Appealing to Theorem~\ref{thm53} with $(\kappa,\chi):=(\aleph_2,\aleph_0)$
provides us with a normal, prolific, streamlined $\aleph_2$-Souslin tree $T$ such that $V^-(T)\cap E^{\aleph_2}_{\aleph_0}=V(T)\cap E^{\aleph_2}_{\aleph_0}=S$.
As $V^-(T)\cap E^{\aleph_2}_{\aleph_0}$ is a nonreflecting stationary set,
Lemma~\ref{cor53}(1) (using $(\varsigma,\chi,\kappa):=(2,\aleph_1,\aleph_2)$) implies that $V(T)\cap E^{\aleph_2}_{\aleph_1}=E^{\aleph_2}_{\aleph_1}$.
\end{proof}
\begin{cor}\label{cor47} Suppose $\ch$ and $\sd_{\aleph_1}$ both hold.
For every stationary $S\s E^{\aleph_2}_{\aleph_0}$,
there exists an $\aleph_2$-Souslin tree $\mathbf T$ such that $V(\mathbf T)$ is a stationary subset of $S$.
\end{cor}
\begin{proof} $\sd_{\aleph_1}$ implies $\square_{\aleph_1}$ which implies that for every stationary $S\s E^{\aleph_2}_{\aleph_0}$ there exists a stationary $R\s S$ that is nonreflecting.
It thus follows from Corollary~\ref{cor46} that for every stationary $S\s E^{\aleph_2}_{\aleph_0}$ there exist a stationary $R\s S$
and an $\aleph_2$-Souslin tree $\mathbf T$ such that $ V(\mathbf T)=R\cup E^{\aleph_2}_{\aleph_1}$.
In addition, $\sd_{\aleph_1}$ yields a uniformly coherent $\aleph_2$-Souslin tree $\mathbf S$ (see \cite[Theorem~7]{MR830071} or \cite[Proposition~2.5 and Theorem~3.6]{paper22}).
By \cite[Remark~2.20]{paper48}, then, $V(\mathbf S)=E^{\aleph_2}_{\aleph_0}$.
Clearly, $\mathbf T+\mathbf S$ is an $\aleph_2$-Souslin tree, and, by Proposition~\ref{sums}(2),
$V(\mathbf T+\mathbf S)=R$.
\end{proof}

\begin{thm}\label{thm54}
Suppose that $\kappa$ is a strongly inaccessible cardinal,
and that $\p(\kappa,\kappa,\sqleftup{S},1,\{S\})$ holds for a given $S\s\acc(\kappa)$.
Then there exists a normal, prolific, streamlined $\kappa$-Souslin tree $T$ such that $V^-(T)=V(T)=S$.
\end{thm}
\begin{proof} The proof is almost identical to that of Theorem~\ref{thm52}, where the only change is
that now, the definition of $T_\alpha$ for a limit $\alpha$ does not explicitly mention the $\mathbf{d}_x^C$'s. Instead, it is:
$$T_\alpha:=\begin{cases}\{\mathbf{b}_x^{C}\mid C\in\mathcal C_\alpha, x\in T\restriction C\},&\text{if }\alpha\in S;\\
\mathcal B(T\restriction\alpha),&\text{otherwise}.\\
\end{cases}$$
The details are left to the reader.
\end{proof}

\begin{cor}\label{cor55}
Suppose that $\kappa$ is a strongly inaccessible cardinal, and $S$ is a nonreflecting stationary subset of $\acc(\kappa)$ on which $\diamondsuit$ holds.
Then there exists a normal prolific streamlined $\kappa$-Souslin tree $T$ such that $V^-(T)=V(T)=S$.
\end{cor}
\begin{proof}  By Theorem~\ref{thm54} together with \cite[Theorem~4.26]{paper23}.
\end{proof}

\section{Realizing all points of some fixed cofinality}\label{sect6}
The main result of this section is Theorem~\ref{thm58} below.
A sample corollary of it reads as follows.

\begin{cor}\label{cor61} In $\mathsf{L}$, for every regular uncountable cardinal $\kappa$ that is not weakly compact,
for every finite nonempty $x\s\reg(\kappa)$ with $\max(x)\le\cf(\sup(\reg(\kappa)))$,
there exists a uniformly homogeneous $\kappa$-Souslin tree $\mathbf T$ such that $V^-(\mathbf T)=\bigcup_{\chi\in x}E^\kappa_\chi$.
\end{cor}
\begin{proof} Work in $\mathsf{L}$.
Let $\kappa$ be regular uncountable cardinal that is not weakly compact,
and let $\langle \chi_i\mid i\le n\rangle$ be a strictly increasing finite sequence of regular cardinals with $\chi_n\le\cf(\sup(\reg(\kappa)))$.

By \cite[Theorem~3.6]{paper22} and \cite[Corollary~4.12]{paper29},
$\p(\kappa,2,{\sq},\kappa,\{E^\kappa_{\ge\chi_n}\})$ holds.
By $\gch$, $\lambda^{<\chi_n}<\kappa$ for all $\lambda<\kappa$.
So, by Theorem~\ref{thm58} below, using $S:={}^{<\kappa}1$,
we may pick a streamlined, normal, $2$-splitting,  uniformly homogeneous,
$\chi_0$-complete, $\chi_0$-coherent, $E^\kappa_{\ge\chi_0}$-regressive $\kappa$-Souslin tree $T^0$.
Furthermore, $T^0$ is $\p^-(\kappa,2,\sq,\kappa,\{E^\kappa_{\geq\chi_n}\})$-respecting.
\begin{claim}\label{claim511} $V^-(T^0)=E^\kappa_{\chi_0}$.
\end{claim}
\begin{proof} Since $T^0$ is $\chi_0$-complete, $V^-(T^0)\cap E^\kappa_{<{\chi_0}}=\emptyset$,
so that $\tr(\kappa\setminus V^-(T^0))$ covers $E^\kappa_{\ge{\chi_0}}$.
By $\gch$, $2^{<\chi_0}<2^{\chi_0}$.
Together with the fact that $T$ is $E^\kappa_{\chi_0}$-regressive,
it follows from Lemma~\ref{cor53}(2) that $E^\kappa_{\chi_0}\s V^-(T^0)$.
Finally, since $T^0$ is ${\chi_0}$-coherent and uniformly homogeneous,
we get from Lemma~\ref{lemma64} below that  $V^-(T^0)\cap E^\kappa_{>{\chi_0}}=\emptyset$.
\end{proof}

If $n=0$, then our proof is complete.  Otherwise, one can continue by recursion,
where the successive step is as follows:
Suppose that $i<n$ is such that $\bigotimes_{j\le i}T^j$
is a streamlined uniformly homogeneous normal $\kappa$-Souslin tree that is
$\p^-(\kappa,2,{\sq},\allowbreak\kappa,\{E^\kappa_{\geq\chi_n}\})$-respecting,
and that $V(\bigotimes_{j\le i}T^j)=\bigcup_{j\le i}E^\kappa_{\chi_j}$.
By Theorem~\ref{thm58} below, using $S:=\bigotimes_{j\le i}T^j$,
we may pick a streamlined, normal, $2$-splitting,  uniformly homogeneous,
$\chi_{i+1}$-complete, $\chi_{i+1}$-coherent, $E^\kappa_{\ge\chi_{i+1}}$-regressive $\kappa$-Souslin tree $T^{i+1}$.
Furthermore, $S\otimes T^{i+1}$ is a normal $\p^-(\kappa,2,\sq,\kappa,\allowbreak\{E^\kappa_{\geq\chi_n}\})$-respecting $\kappa$-Souslin tree.
By an analysis similar to that of Claim~\ref{claim511}, $V^-(T^{i+1})=E^\kappa_{\chi_{i+1}}$.
Therefore, $\bigotimes_{j\le i+1}T^j$
is a uniformly homogeneous normal $\kappa$-Souslin tree that is
$\p^-(\kappa,2,{\sq},\allowbreak\kappa,\{E^\kappa_{\geq\chi_n}\})$-respecting.
In addition, by Proposition~\ref{products}(2), $V(\bigotimes_{j\le i+1}T^j)=\bigcup_{j\le i+1}E^\kappa_{\chi_j}$.
\end{proof}

We start by giving a definition.

\begin{defn}\label{uniformly}A streamlined $\kappa$-tree $T$  is
\emph{$\chi$-coherent} iff for all $s,t\in T$,
$\{ \xi\in\dom(s)\cap\dom(t)\mid s(\xi)\neq t(\xi)\}$ has size $<\chi$.
\end{defn}

\begin{lemma}\label{lemma64} Suppose that
$\chi<\kappa$ is a cardinal,
and that $T$ is a streamlined, $\chi$-coherent uniformly homogeneous $\kappa$-tree.
Then $V^-(T)\s E^\kappa_{\le\chi}$.
\end{lemma}
\begin{proof} Let $\alpha\in E^\kappa_{>\chi}$. Suppose that $B\s T$ is an $\alpha$-branch, and we shall show it is not vanishing.

For every $\beta<\alpha$, let $t_\beta$ denote the unique element of $T_\beta\cap B$.
Fix a node $t\in T_\alpha$.
For every $\beta\in E^\alpha_\chi$, by $\chi$-coherence, the following ordinal is smaller than $\beta$:
$$\epsilon_\beta:=\sup\{ \xi<\beta\mid t_\beta(\xi)\neq t(\xi)\}.$$
As $\cf(\alpha)>\chi$, $E^\alpha_\chi$ is a stationary subset of $\alpha$,
so we may fix a large enough $\epsilon<\alpha$ for which $R:=\{\beta\in E^\alpha_\chi\mid\epsilon_\beta<\epsilon\}$ is stationary.
As $T$ is uniformly homogeneous, $t_\epsilon*t$ is in $T_\alpha$.
For every $\beta\in R$, $t_\beta=(t_\epsilon*t)\restriction\beta$.
But since $R$ is cofinal in $\alpha$, it is the case that $t_\epsilon*t$ constitutes a limit for $B$.
Therefore, $B$ is not vanishing.
\end{proof}

In the context of streamlined $\kappa$-trees, there is a neater way of presenting the operation of product (compare with Definition~\ref{def-classical-product-tree}):
\begin{defn}[{\cite[\S6.7]{paper23}}]\label{def54}
For every function $x:\alpha\rightarrow{}^\tau H_\kappa$ and every $i<\tau$,
we let $(x)_i:\alpha\rightarrow H_\kappa$ be $\langle x(\beta)(i)\mid \beta<\alpha\rangle$.
Using this notation, for every sequence $\langle T^i\mid i<\tau\rangle$ of streamlined $\kappa$-trees,
one may identify $\bigotimes_{i<\tau}T^i$ with the streamlined tree $T:=\{ x\in{}^{<\kappa}({}^\tau H_\kappa)\mid \forall i<\tau\,[(x)_i\in T^i]\}$.
\end{defn}
\begin{remark} The product of two uniformly homogeneous $\kappa$-trees is uniformly homogeneous.
\end{remark}

Before we can state the main result of this section, we need one more definition.
\begin{defn}[\cite{paper20}]\label{respecting}
A streamlined $\kappa$-tree $X$
is \emph{$\p_\xi^-(\kappa, \mu,\mathcal R, \theta, \mathcal S)$-respecting}
if there exists a subset $\S\s\kappa$ and a sequence of mappings
$\langle d ^C:(X\restriction C)\rightarrow {}^\alpha H_\kappa\cup\{\emptyset\}\mid \alpha<\kappa, C\in\mathcal C_\alpha\rangle$ such that:
\begin{enumerate}
\item\label{respectingonto} for all $\alpha\in\S$ and $C\in\mathcal C_\alpha$, $X_\alpha\s \im(d^C)$;
\item $\vec{\mathcal C}=\langle \mathcal C_\alpha\mid\alpha<\kappa\rangle$ witnesses $\p_\xi^-(\kappa, \mu,\mathcal R, \theta, \{S\cap\S\mid S\in \mathcal S\})$;
\item \label{respectcohere} for all sets $D\sq C$ from $\vec{\mathcal C}$ and $x\in X\restriction D$, $d^D(x)=d^C(x)\restriction\sup(D)$.
\end{enumerate}
\end{defn}
\begin{remark}\label{rmk58} \begin{enumerate}
\item If $\p_\xi^-(\kappa, \mu,\mathcal R, \theta, \mathcal S)$ holds,
then the normal streamlined $\kappa$-tree $X:={}^{<\kappa}1$ is $\p_\xi^-(\kappa, \mu,\mathcal R, \theta, \mathcal S)$-respecting;
\item If $\kappa=\lambda^+$ for an infinite regular cardinal $\lambda$,
and $\p_\lambda^-(\kappa, \mu,\allowbreak\sqleft{\lambda}, \theta, \{E^\kappa_\lambda\})$ holds,
then every $\kappa$-tree is $\p_\lambda^-(\kappa, \mu,\allowbreak\sqleft{\lambda}, \theta, \{E^\kappa_\lambda\})$-respecting.
\end{enumerate}
\end{remark}

\begin{lemma}\label{lemma58} Suppose that:
\begin{itemize}
\item $X$ is a streamlined $\kappa$-tree that is $\p_\xi^-(\kappa, \mu,\mathcal R, \kappa, \mathcal S)$-respecting,
as witnessed by some $\vec{\mathcal C}$ and $\S$;
\item $Y$ is a streamlined $\kappa$-tree that is $\p_\xi^-(\kappa, \mu,\mathcal R, \kappa, \{S\cap\S\mid S\in \mathcal S\})$-respecting,
as witnessed by the same $\vec{\mathcal C}$.
\end{itemize}

Then the product $X\otimes Y$ is $\p_\xi^-(\kappa, \mu,\mathcal R, \kappa, \mathcal S)$-respecting.
\end{lemma}
\begin{proof}  In view of Definition~\ref{def54}, for every two functions $x,y$ from an ordinal $\alpha<\kappa$ to $H_\kappa$,
we denote by $\myceil{(x,y)}$ the unique function $p:\alpha\rightarrow{}^2H_\kappa$ such that $(p)_0=x$ and $(p)_1=y$.
Note that $X\otimes Y=\bigcup_{\alpha<\kappa}\{ \myceil{(x,y)}\mid (x,y)\in X_\alpha\times Y_\alpha\}$.

Write $\vec{\mathcal C}$ as $\langle \mathcal C_\alpha\mid\alpha<\kappa\rangle$.
Fix a sequence of mappings
$\langle d^C:(X\restriction C)\rightarrow {}^\alpha H_\kappa\cup\{\emptyset\}\mid \alpha<\kappa, C\in\mathcal C_\alpha\rangle$ such that:
\begin{enumerate}
\item for all $\alpha\in\S$ and $C\in\mathcal C_\alpha$, $X_\alpha\s \im(d^C)$;
\item $\vec{\mathcal C}=\langle \mathcal C_\alpha\mid\alpha<\kappa\rangle$ witnesses $\p_\xi^-(\kappa, \mu,\mathcal R, \kappa, \{S\cap\S\mid S\in \mathcal S\})$;
\item for all sets $D\sq C$ from $\vec{\mathcal C}$ and $x\in X\restriction D$, $d^D(x)=d^C(x)\restriction\sup(D)$.
\end{enumerate}

Fix a stationary $\S'\s\S$ and a sequence of mappings
$\langle e^C:(Y\restriction C)\rightarrow {}^\alpha H_\kappa\cup\{\emptyset\}\mid \alpha<\kappa, C\in\mathcal C_\alpha\rangle$ such that:
\begin{enumerate}
\item[(4)] for all $\alpha\in\S'$ and $C\in\mathcal C_\alpha$, $Y_\alpha\s \im(e^C)$;
\item[(5)] $\vec{\mathcal C}=\langle \mathcal C_\alpha\mid\alpha<\kappa\rangle$ witnesses $\p_\xi^-(\kappa, \mu,\mathcal R, \kappa, \{S\cap\S'\mid S\in \mathcal S\})$;
\item[(6)] for all sets $D\sq C$ from $\vec{\mathcal C}$ and $y\in Y\restriction D$, $e^D(y)=e^C(y)\restriction\sup(D)$.
\end{enumerate}

Let $\vec B=\langle B_{x,y}\mid (x,y)\in X\times Y\rangle$ be a partition of $\kappa$ into cofinal subsets of $\kappa$.
Define a sequence of mappings
$\langle b^C:(X\otimes Y)\restriction C\rightarrow {}^\alpha H_\kappa\cup\{\emptyset\}\mid \alpha<\kappa, C\in\mathcal C_\alpha\rangle$,
as follows. Let $\alpha<\kappa$ and $C\in\mathcal C_\alpha$.

$\br$  For every $\beta\in C$, if there are $x\in X\restriction(C\cap\beta)$ and $y\in Y\restriction(C\cap\beta)$ such that $\beta\in B_{x,y}$,
then since $\vec B$ is a sequence of pairwise disjoint sets, this pair $(x,y)$ is unique,
and we let $b^C(p):=\myceil{(d^C(x),e^C(y))}$ for every $p\in(X\otimes Y)_\beta$.

$\br$ For every $\beta\in C$ for which there is no such pair $(x,y)$, we let $b^C(p):=\emptyset$ for every $p\in(X\otimes Y)_\beta$.

\begin{claim} Suppose $D\sq C$ are sets from $\vec{\mathcal C}$.
For every $p\in (X\otimes Y)\restriction D$, $b^D(p)=b^C(p)\restriction\sup(D)$.
\end{claim}
\begin{proof} Given $p\in (X\otimes Y)\restriction D$. Denote $\beta:=\dom(p)$.
Note that $D\cap\beta=C\cap\beta$. Now, there are two options:

$\br$  There are $x\in X\restriction(C\cap\beta)$ and $y\in Y\restriction(C\cap\beta)$ such that $\beta\in B_{x,y}$.
Then $b^D(p)=\myceil{(d^D(x),e^D(y))}$ and $b^C(p)=\myceil{(d^C(x),e^C(y))}$.
Since $D\sq C$, we know that $d^D(x)=d^C(x)\restriction\sup(D)$ and $e^D(y)=e^C(y)\restriction\sup(D)$.
Therefore, $b^D(p)=d^C(p)\restriction\sup(D)$.

$\br$ There are no such $x$ and $y$. Then $b^D(p)=\emptyset=d^C(p)$.
\end{proof}

Consider the following set:
$$\S'':=\{\alpha\in\S'\mid \forall C\in\mathcal C_\alpha\forall x\in(X\restriction\alpha)\forall y\in(Y\restriction\alpha)\,[\sup( C_\alpha\cap B_{x,y})=\alpha]\}.$$

\begin{claim}  $\vec{\mathcal C}=\langle \mathcal C_\alpha\mid\alpha<\kappa\rangle$ witnesses $\p_\xi^-(\kappa, \mu,\mathcal R, \kappa, \{S\cap\S''\mid S\in \mathcal S\})$.
\end{claim}
\begin{proof}  Let $\langle B_i\mid i<\kappa\rangle$ be a given sequence of cofinal subsets of $\kappa$.
Let $\pi:\kappa\leftrightarrow\kappa\uplus(X\times Y)$ be a surjection.
As $X$ and $Y$ are $\kappa$-tree, the set $D:=\{ \alpha<\kappa\mid \pi[\alpha]=\alpha\uplus((X\restriction\alpha)\times(Y\restriction\alpha))\}$ is a club in $\kappa$.
By Clause~(5), then,
for every $S\in\mathcal{S}$, there are stationarily many $\alpha\in S\cap\S'\cap D$
such that for all $C\in\mathcal C_\alpha$ and $i<\alpha$, $\sup(\nacc(C)\cap B_{\pi(i)})=\alpha$.
In particular, for every $S\in\mathcal{S}$, there are stationarily many $\alpha\in S\cap\S''$
such that for all $C\in\mathcal C_\alpha$ and $i<\alpha$, $\sup(\nacc(C)\cap B_i)=\alpha$.
\end{proof}

\begin{claim} Let $\alpha\in\S''$ and $C\in\mathcal C_\alpha$. Then $(X\otimes Y)_\alpha\s \im(b^C)$.
\end{claim}
\begin{proof}  Let $(s,t)\in X_\alpha \times Y_\alpha$. As $\S''\s\S'\s S$, using Clauses (1) and (4), we may fix $x\in X\restriction C$ and $y\in Y\restriction C$
such that $d^C(x)=s$ and $e^c(y)=t.$ As $\alpha\in\S''$, we may pick $\beta\in C_\alpha\cap B_{x,y}$ above $\max\{\dom(x),\dom(y)\}$.
Let $p$ be an arbitrary element of $(X\otimes Y)\restriction C$. Then $b^C(p):=\myceil{(d^C(x),e^C(y))}=\myceil{(s,t)}$.
\end{proof}

This completes the proof.
\end{proof}
\begin{thm}\label{thm58}
Suppose that:
\begin{itemize}
\item $\varsigma<\kappa$ is a cardinal;
\item $\nu\le\chi<\kappa$ are cardinals such that $\lambda^{<\chi}<\kappa$ for all $\lambda<\kappa$;
\item $S$ is a $\p^-(\kappa,2,{\sqleft\nu},\kappa,\{E^\kappa_{\geq\chi}\})$-respecting streamlined normal $\kappa$-tree with no $\kappa$-sized antichains;
\item $\diamondsuit(\kappa)$ holds.
\end{itemize}
Then there exists a streamlined, normal, $\varsigma$-splitting, prolific, uniformly homogeneous,
$\chi$-complete, $\chi$-coherent, $E^\kappa_{\ge\chi}$-regressive $\kappa$-Souslin tree $T$ such that
$S\otimes T$ is a normal $\p^-(\kappa,2,{\sqleft\nu},\kappa,\{E^\kappa_{\geq\chi}\})$-respecting $\kappa$-Souslin tree.
\end{thm}
\begin{proof} Fix a stationary $\S\s\kappa$ and a sequence $\langle d^\alpha:S\restriction C_\alpha\rightarrow {}^\alpha H_\kappa\cup\{\emptyset\}\mid\alpha<\kappa\rangle$ such that:
\begin{enumerate}
\item\label{respectingonto} for all $\alpha\in\S$, $S_\alpha\s \im(d^\alpha)$;
\item $\vec C:=\langle C_\alpha\mid\alpha<\kappa\rangle$ witnesses $\p^-(\kappa, 2,{\sqleft\nu}, \kappa, \{\S\})$;
\item for all $\alpha<\beta<\kappa$, if $C_\alpha\sq C_\beta$, then $d^\alpha(x)=d^\beta(x)\restriction\alpha$ for every $x\in S\restriction C_\alpha$.
\end{enumerate}

Without loss of generality, we may assume that $0\in C_\alpha$ for all nonzero $\alpha<\kappa$.

The upcoming construction follows the proof of \cite[Proposition~2.5]{paper22}.
Let $\langle R_i \mid i<\kappa\rangle$ and $\langle \Omega_\beta\mid\beta<\kappa\rangle$ together witness $\diamondsuit(H_\kappa)$.
Let $\pi:\kappa\rightarrow\kappa$ be such that $\alpha\in R_{\pi(\alpha)}$ for all $\alpha<\kappa$.
From $\diamondsuit(\kappa)$, we have $\left|H_\kappa\right| =\kappa$,
thus let $\lhd$ be some well-ordering of $H_\kappa$ of order-type $\kappa$,
and let $\phi:\kappa\leftrightarrow H_\kappa$ witness the isomorphism $(\kappa,\in)\cong(H_\kappa,\lhd)$.
Put $\psi:=\phi\circ\pi$.

We now recursively construct a sequence $\langle T_\alpha\mid \alpha<\kappa\rangle$ of levels
whose union will ultimately be the desired tree $T$.
Let $T_0:=\{\emptyset\}$, and for all $\alpha<\kappa$, let $$T_{\alpha+1}:=\{ t{}^\smallfrown\langle i\rangle\mid t\in T_\alpha, i<\max\{\varsigma,\omega,\alpha\}\}.$$
Next, suppose that $\alpha\in\acc(\kappa)$, and that $\langle T_\beta\mid \beta<\alpha\rangle$ has already been defined.
We shall identify some $\mathbf b^\alpha\in\mathcal B(T\restriction\alpha)$, and then define the $\alpha^{\text{th}}$-level, as follows:
\begin{equation}\tag{$\star$}\label{promise3}
T_\alpha:=\begin{cases}
\mathcal B(T\restriction\alpha),&\text{if }\alpha\in E^\kappa_{<\chi};\\
\{x*\mathbf b^\alpha\mid x\in T\restriction\alpha\},&\text{if }\alpha\in E^\kappa_{\ge\chi}.
\end{cases}
\end{equation}

We shall obtain $\mathbf b^\alpha$ as a limit $\bigcup\im(b^\alpha)$ of a sequence $b^\alpha\in\prod_{\beta\in C_\alpha}T_\beta$ that we define recursively, as follows.
Let $b^\alpha(0):=\emptyset$.
Next, suppose $\beta^-<\beta$ are two successive points of $C_\alpha$, and that $b^\alpha(\beta^-)$ has already been defined.
There are two possible options:

$\br$ If $\psi(\beta)$ happens to be a pair $(y,x)$ lying in $(S\restriction \beta^-)\times (T\restriction\beta^-)$,
and the following set happens to be nonempty:
$$Q^{\alpha,\beta}:=\{t\in T_\beta\mid \exists(\bar s,\bar t)\in \Omega_\beta\,[\bar s\s d^\alpha(y)\restriction \beta\ \&\   (\bar t\cup(x*b^\alpha(\beta^-)))\s t] \},$$
then let $t$ denote its $\lhd$-least element, and put $b^\alpha(\beta):=b^\alpha(\beta^-)* t$.

$\br$ Otherwise, let $b^\alpha(\beta)$ be the $\lhd$-least element of $T_\beta$ that extends $b^\alpha(\beta^-)$.

\medskip

As always, for all $\beta \in \acc(C_\alpha)$ such that $b^\alpha\restriction\beta$ has already been defined,
we let $b^\alpha(\beta):=\bigcup\im(b^\alpha\restriction\beta)$ and infer that it belongs to $T_\beta$.
Indeed, either $\cf(\beta)<\chi$, and then $b^\alpha(\beta)\in\mathcal B(T\restriction\beta)=T_\beta$,
or $\cf(\beta)\ge\chi\ge\nu$, and then $C_\beta=C_\alpha\cap\beta$ from which it follows that $b^\alpha(\beta)=\mathbf b^\beta\in T_\beta$.
This completes the definition of $b^\alpha$, hence also that of $\mathbf b^\alpha$.
Finally, let $T_\alpha$ be defined as promised in \eqref{promise3}.

It is clear that  $T := \bigcup_{\alpha < \kappa} T_\alpha$
is a streamlined, normal, $\varsigma$-splitting, prolific, uniformly homogeneous,
$\chi$-complete $\kappa$-tree.

\begin{claim} $T$ is $\chi$-coherent.
\end{claim}
\begin{proof} Suppose not, and let $\alpha$ be the least ordinal to accommodate
$s,t\in T_\alpha$ such that $s$ differs from $t$ on a set of size $\ge\chi$.
Clearly, $\alpha\in E^\kappa_{\ge\chi}$.
So $s=x*\mathbf b^\alpha$ and $t=y*\mathbf b^\alpha$
for nodes $x,y\in T\restriction\alpha$,
and hence $x$ and $y$ differ on a set of size $\ge\chi$, contradicting the minimality of $\alpha$.
\end{proof}

\begin{claim} $T$ is $E^\kappa_{\ge\chi}$-regressive.
\end{claim}
\begin{proof} To define $\rho:T\restriction E^\kappa_{\ge\chi}\rightarrow T$,
let $\alpha\in E^\kappa_{\ge\chi}$.
By the definition of $T_\alpha$, for every $t\in T$, there exists some $x\in T\restriction\alpha$ such that $t=x*\mathbf b^\alpha$,
so we let $\rho(t)$ be an element of $T\restriction\alpha$ such that $t=\rho(t)*\mathbf b^\alpha$.
Now, if $s,t\in T_\alpha$ are such that $\rho(t)\s s$ and $\rho(s)\s t$,
then $\rho(t)\s \rho(s)*\mathbf b^\alpha$ and $\rho(s)\s \rho(t)*\mathbf b^\alpha$. In particular, $\rho(s)$ is compatible with $\rho(t)$.
Without loss of generality, $\rho(s)\s \rho(t)$. Then $t=\rho(s)*\mathbf b^\alpha=s$.
\end{proof}
\begin{claim} $T$ is $\p^-(\kappa,2,{\sqleft\nu},\kappa,\{\S\})$-respecting, as witnessed by $\vec C$.
\end{claim}
\begin{proof} Define $\langle e^\alpha:T\restriction C_\alpha\rightarrow T_\alpha\mid\alpha<\kappa\rangle$ via:
$$e^\alpha(x):=x*\mathbf b^\alpha.$$
The second part of \eqref{promise3} implies that $S_\alpha=\im(d^\alpha)$ for all $\alpha\in E^\kappa_{\ge\chi}\supseteq\S$.
In addition, it is clear that for all $\alpha<\beta<\kappa$, if $C_\alpha\sq C_\beta$, then ${\mathbf b}^\alpha={\mathbf b}^\beta\restriction\alpha$,
and hence $e^\alpha(x)=e^\beta(x)\restriction\alpha$  for every $x\in S\restriction C_\alpha$.
\end{proof}

It thus follows from Lemma~\ref{lemma58} that $S\otimes T$ is $\p^-(\kappa,2,{\sqleft\nu},\kappa,\{E^\kappa_{\geq\chi}\})$-respecting.
It is clear that $S\otimes T$ is normal, thus we are left with verifying that it is Souslin.
To this end, let $A$ be a maximal antichain in $S\otimes T$.
As both $S$ and $T$ are normal, it follows that for every $z\in T$, the following (upward-closed) set is cofinal in $S$:
$$D_z:=\{s\in S\mid \exists (\bar s,\bar t)\in A \exists t\in T\cap z^\uparrow\,[\dom(s)=\dom(t), \bar s\s s, \bar t\s t]\}.$$

As an application of $\diamondsuit(H_\kappa)$,
using the parameter $p:=\{\phi,S\otimes T, A, \langle D_z\mid z\in T\rangle\}$,
we get that for every $i<\kappa$, the following set is cofinal (in fact, stationary) in $\kappa$:
$$B_i:=\{\beta\in R_i\mid \exists \mathcal M\prec H_{\kappa^+}\,(p\in \mathcal M, \mathcal M\cap\kappa=\beta, \Omega_\beta=A\cap \mathcal M)\}.$$

Note that  	$(S\restriction\beta)\otimes (T\restriction\beta)\s \phi[\beta]$ for every $\beta\in \bigcup_{i<\kappa}B_i$.
Now, as $\vec C$ witnesses $\p^-(\kappa, 2,{\sqleft\nu}, \kappa, \{\S\})$,
we may fix some $\alpha\in \S$ such that, for all $i<\alpha$, $$\sup (\nacc(C_\alpha)\cap B_i)=\alpha.$$
In particular, $(S\restriction\alpha)\otimes (T\restriction\alpha)\s \phi[\alpha]$.
As $\alpha\in\S$, we also know that $S_\alpha\s \im(d^\alpha)$ and that $\cf(\alpha)\ge\chi$.

\begin{claim} $A\s(S\otimes T)\restriction\alpha$. In particular, $|A|<\kappa$.
\end{claim}
\begin{proof} As $A$ is an antichain, it suffices to prove that every element of $(S\otimes T)_\alpha$ extends some element of $A$.
To this end, fix $(s',t')\in (S\otimes T)_\alpha$.
Since $S_\alpha\s \im(d^\alpha)$,
we may fix a $y\in S\restriction C_\alpha$ such that $d^\alpha(y)=s'$.
Recalling \eqref{promise3}, we may also fix some $x\in T\restriction C_\alpha$ such that $t'=x*\mathbf b^\alpha$.

As the pair $(y,x)$ is an element of $(S\restriction\alpha)\times (T\restriction\alpha)$, we may find an $i<\alpha$ such that $\phi(i)=(y,x)$,
and then find a $\beta\in \nacc(C_\alpha)\cap B_i$ such that $\beta^-:=\sup(C_\alpha\cap\beta)$ is greater than $\max\{\dom(y),\dom(x)\}$.
Note that $\psi(\beta)=\phi(\pi(\beta))=\phi(i)=(y,x)$.
\begin{subclaim} $\Omega_\beta=A\cap ((S\otimes T)\restriction\beta)$,
and  $Q^{\alpha,\beta}\neq\emptyset$.
\end{subclaim}
\begin{proof}
As $\beta\in B_i$, we may fix $\mathcal M\prec H_{\kappa^+}$ such that all of the following hold:
\begin{itemize}\item $\{\phi,S\otimes T, A, \langle D_x\mid x\in T\rangle\}\in \mathcal M$;
\item $\mathcal M\cap\kappa=\beta$;
\item $\Omega_\beta=A\cap \mathcal M$
\end{itemize}

By elementarity, $(T\otimes S)\cap \mathcal M=(S\otimes T)\restriction\beta$,
and
$\Omega_\beta=A\cap \mathcal M=A\cap ((S\otimes T)\restriction\beta)$.
Then $z:=t'\restriction\beta^-$ is in $\mathcal M$,
and hence, so is $D_{z}$.
Pick in $\mathcal M$ a maximal antichain $\bar D$ in $D_z$.
Since $D_z$ is cofinal in $S$, $\bar D$ is a maximal antichain in $S$.
Since $S$ has no $\kappa$-sized antichains, we may find a large enough $\gamma\in\mathcal M\cap\kappa$
such that $\bar D\s S\restriction\gamma$. It thus follows that $s'\restriction \gamma$ extends an element of $\bar D$,
but since $D_z$ is upward-closed, $s:=s'\restriction\gamma$ is in $D_z$.
It follows that we may fix $(\bar s,\bar t)\in A$ and $t\in T_\gamma\cap z^\uparrow$ such that $\bar s\s s$ and $\bar t\s t$.
As $\Omega_\beta=A\cap ((S\otimes T)\restriction\beta)$, $(d^\alpha(y)\restriction \beta)\restriction\gamma=s$
and $x*b^\alpha(\beta^-)=z\s t$, we infer that $t\in Q^{\alpha,\beta}$.
\end{proof}

It follows that $b^\alpha(\beta)=b^\alpha(\beta^-)* t$
for some $t\in Q^{\alpha,\beta}$. This means that we may pick $(\bar s,\bar t)\in \Omega_\beta\s A$
such that $\bar s\s s'\restriction \beta$ and $\bar t\cup(x*b^\alpha(\beta^-))\s t$.
Therefore, $\bar t\s x*b^\alpha(\beta)$. Altogether, $(\bar s,\bar t)\in A$, $\bar s\s s'$ and $\bar t\s t'$.
\end{proof}
This completes the proof.
\end{proof}

We now arrive at the proof of Theorem~\ref{thma}:
\begin{thm}\label{thm511} We have $(1)\implies(2)\implies(3)$:
\begin{enumerate}
\item there exists a $\kappa$-Souslin tree $\mathbf T$ such that $V(\mathbf T)=\emptyset$;
\item there exists a normal and splitting $\kappa$-tree $\mathbf T$ such that $V(\mathbf T)$ is nonstationary;
\item $\kappa$ is not the successor of a cardinal of countable cofinality.
\end{enumerate}

In addition, in $\mathsf{L}$, for $\kappa$ not weakly compact, $(3)\implies(1)$.
\end{thm}
\begin{proof} $(1)\implies(2)$: If $\mathbf T=(T,<_T)$ is a $\kappa$-Souslin tree, then a standard argument (see \cite[Lemma~2.4]{paper20}) shows that for some club $D\s\kappa$,
$\mathbf T'=(T\restriction D,{<_T})$ is normal and splitting.  Clearly, if $V(\mathbf T)=\emptyset$, then $V(\mathbf T')=\emptyset$, as well.

$(2)\implies(3)$: Suppose that $\mathbf T$  is a normal and splitting $\kappa$-tree.
If $\kappa$ is the successor of a cardinal of countable cofinality then by Corollary~\ref{cor27}, $V(\mathbf T)$ covers the stationary set $E^\kappa_\omega$.

Hereafter, work in $\mathsf L$, and suppose that $\kappa$ is a regular uncountable cardinal that is not weakly compact and not the successor of a cardinal of countable cofinality.
Then by Corollary~\ref{cor61} together with Proposition~\ref{prop22}(2) there are $\kappa$-Souslin trees $\mathbf T^0,\mathbf T^1$ such that $V(\mathbf T^0)=E^\kappa_\omega$ and $V(\mathbf T^1)=E^\kappa_{\omega_1}$.
The disjoint sum of the two $\mathbf T:=\sum\{\mathbf T^0,\mathbf T^1\}$ is clearly $\kappa$-Souslin.
In addition, by Proposition~\ref{sums}(2), $V(\mathbf T)=V(\mathbf T^0)\cap V(\mathbf T^1)=\emptyset$.
\end{proof}
\begin{remark} The $\kappa$-Souslin tree $\mathbf T$ constructed in the preceding proof satisfies $V(\mathbf T)=\emptyset$,
yet it has a $\kappa$-Souslin subtree $\mathbf T'$ for which $V(\mathbf T')$ is stationary.
A $\kappa$-tree $\mathbf T$ is said to be \emph{full} iff for every $\alpha\in\acc(\kappa)$,
there is no more than one vanishing $\alpha$-branch in $\mathbf T$.
It is clear that if $\mathbf T$ is a full $\kappa$-tree that is splitting (resp.~Aronszajn), then $V(\mathbf T)$ is empty (resp.~nonstationary).
In \cite{paper62}, we construct full $\kappa$-Souslin trees,
thus giving an example of a $\kappa$-Souslin tree $\mathbf T$ such that $V(\mathbf T')$ is nonstationary for all of its $\kappa$-subtrees $\mathbf T'$.
\end{remark}

We conclude this section by pointing out that by
using \cite[Theorem~3.6]{paper22} and a proof similar to that of Theorem~\ref{thm511}, we get more information on the model studied in Corollary~\ref{cor47}.

\begin{cor} Suppose that $\ch$ and $\sd_{\aleph_1}$ both hold.
Then there are $\aleph_2$-Souslin trees $\mathbf T^0,\mathbf T^1,\mathbf T^2,\mathbf T^3$ such that:
\begin{itemize}
\item $V(\mathbf T^0)=\emptyset$;
\item $V(\mathbf T^1)=E^{\aleph_2}_{\aleph_0}$;
\item $V(\mathbf T^2)=E^{\aleph_2}_{\aleph_1}$;
\item $V(\mathbf T^3)=\acc(\aleph_2)$. \qed
\end{itemize}
\end{cor}

\section{Souslin trees with an ascent path}\label{sect7}

The subject matter of this section is the following definition.
\begin{defn}[Laver]  Suppose that $\mathbf{T} = (T, <_T)$ is a tree of some height $\kappa$.
A \emph{$\mu$-ascent path} through $\mathbf{T}$ is a sequence
$\vec f=\langle f_\alpha \mid \alpha < \kappa \rangle$ such that:
\begin{itemize}
\item for every $\alpha < \kappa$, $f_\alpha:\mu \rightarrow T_\alpha$ is a function;
\item for all $\alpha < \beta < \kappa$, there is an $i < \mu$ such that $f_\alpha(j) <_T f_\beta(j)$ whenever $i\le j<\mu$.
\end{itemize}
\end{defn}

We will show that Souslin trees having a large set of vanishing levels
are compatible with carrying an ascent path.
For this, we shall make use of the following strengthening of $\p_\xi^-(\kappa,\mu^+,{\sq},\theta,\mathcal{S})$:

\begin{defn}[{\cite[\S4.6]{paper23}}]\label{indexedP}
The principle $\p_\xi^-(\kappa, \mu^{\ind},\allowbreak{\sq},\theta,\mathcal{S})$ asserts the existence of a $\xi$-bounded $\mathcal C$-sequence $\langle \mathcal{C}_\alpha\mid \alpha<\kappa\rangle$
together with a sequence $\langle i(\alpha)\mid \alpha<\kappa\rangle$ of ordinals in $\mu$, such that:
\begin{itemize}
\item for every $\alpha<\kappa$, there exists a canonical enumeration $\langle C_{\alpha,i}\mid i(\alpha)\le i<\mu\rangle$ of $\mathcal C_\alpha$
satisfying that the sequence $\langle \acc({C}_{\alpha,i})\mid i(\alpha)\le i<\mu\rangle$
is $\s$-increasing with $\bigcup_{i\in[i(\alpha),\mu)}\acc(C_{\alpha,i})=\acc(\alpha)$;
\item for all $\alpha<\kappa$, $i\in[i(\alpha),\mu)$ and $\bar\alpha\in\acc({C}_{\alpha,i})$, it is the case that $i\ge i(\bar\alpha)$ and $C_{\bar\alpha,i}\sq C_{\alpha,i}$;
\item for every sequence $\langle B_\tau\mid \tau<\theta\rangle$ of cofinal subsets of $\kappa$, and every $S\in\mathcal{S}$, there are stationarily many $\alpha\in S$
such that for all $C\in\mathcal C_\alpha$ and $\tau<\min\{\alpha,\theta\}$, $\sup(\nacc(C)\cap B_\tau)=\alpha$.
\end{itemize}
\end{defn}

Conventions \ref{proxydef2} and \ref{conv35} apply to the preceding, as well.

\begin{lemma}\label{l63}
Suppose that:
\begin{itemize}
\item $\mu<\kappa$ is an infinite cardinal;
\item $K$ is a streamlined $\kappa$-tree;
\item $\p(\kappa, \mu^{\ind},\allowbreak{\sq},1)$ holds.
\end{itemize}

Then there exists a normal and splitting streamlined $\kappa$-Souslin tree $T$ with $V(T)\supseteq V^-(K)$
such that $T$ admits a $\mu$-ascent path.
\end{lemma}
\begin{proof} As a preparatory step, we shall need the following simple claim.
\begin{claim} We may assume that $\mathcal B(K)\neq\emptyset$.
\end{claim}
\begin{proof} For every $\eta\in K$, define a function $\eta':\dom(\eta)\rightarrow H_\kappa$ via $\eta'(\alpha):=(\eta(\alpha),0)$.
Then $K':=\{ \eta'\mid \eta\in K\}\uplus{}^{<\kappa}1$ is a streamlined $\kappa$-tree with $V^-(K')=V^-(K)$ and, in addition, $\mathcal B(K')\neq\emptyset$.
\end{proof}

Let $\cvec C=\langle \mathcal{C}_\alpha\mid \alpha<\kappa\rangle$
and $\langle i(\alpha)\mid \alpha<\kappa\rangle$ witness together that $\p^-(\kappa, \mu^{\ind},\allowbreak{\sq},1)$ holds.
In particular, $\cvec C$ is a $\p^-(\kappa,\kappa,{\sq},1)$-sequence
satisfying that, for all $\alpha\in\acc(\kappa)$ and $C,D\in\mathcal C_\alpha$,  $\sup(C\cap D)=\alpha$.
As always, we may also assume that $0\in \bigcap_{0<\alpha<\kappa}\bigcap\mathcal C_\alpha$.

Using $\cvec C$ and $K$,
construct the sequence of levels $\langle T_\alpha\mid \alpha<\kappa\rangle$ exactly as in the proof of Theorem~\ref{thm41},
so that $T:=\bigcup_{\alpha<\kappa}T_\alpha$ is a normal and splitting streamlined $\kappa$-Souslin tree.
From Claim~\ref{c372}, we infer that $V(T)\supseteq V^-(K)$.

In addition, the construction of Theorem~\ref{thm41} ensures that for every $\alpha\in \acc(\kappa)$, it is the case that
$$T_\alpha=\{\mathbf b_x^{C,\eta}\mid C\in\mathcal C_\alpha,\eta\in K_\alpha, x\in T\restriction C\}.$$

Fix $\zeta\in\mathcal B(K)$. For every $\alpha\in\acc(\kappa)$,
using the canonical enumeration $\langle C_{\alpha,i}\mid i(\alpha)\le i<\mu\rangle$ of $\mathcal C_\alpha$,
we define a function $f_\alpha:\mu\rightarrow T_\alpha$ via
$$f_\alpha(j):=\mathbf{b}_\emptyset^{C_{\alpha,\max\{j,i(\alpha)\}},\zeta\restriction\alpha}.$$

\begin{claim} Let $\beta<\alpha$ be a pair of ordinals in $\acc(\kappa)$.
Then there exists an $i<\mu$ such that $f_\beta(j)\s f_\alpha(j)$ whenever $i\le j<\mu$.
\end{claim}
\begin{proof} Note that by Claim~\ref{c371}, for all $C\in\mathcal C_\alpha$, $\eta\in K_\alpha$, and $x\in T\restriction(C\cap\beta)$,
if $\beta\in\acc(C)$, then $\mathbf b_x^{C,\eta}\restriction\beta=\mathbf b_x^{C\cap\beta,\eta\restriction\beta}$.

Now, by Definition~\ref{indexedP}, we may fix a large enough $i\in[i(\alpha),\mu)$ such that $\beta\in\acc(C_{\alpha,j})$ whenever $i\le j<\mu$.
Let $j$ be such an ordinal. Then $j\ge i(\beta)$ and $C_{\alpha,j}\cap\beta=C_{\beta,j}$, so that
$$f_\beta(j)=\mathbf{b}_\emptyset^{C_{\beta,j},\zeta\restriction\beta}=\mathbf{b}_\emptyset^{C_{\alpha,j},\zeta\restriction\alpha}\restriction\beta=f_\alpha(j)\restriction\beta,$$
as sought.
\end{proof}

It now easily follows that $T$ admits a $\mu$-ascent path.
\end{proof}

\begin{cor} Suppose that:
\begin{itemize}
\item $\lambda$ is an uncountable cardinal satisfying $\square_\lambda$ and $2^\lambda=\lambda^+$;
\item $\mu<\lambda$ is an infinite regular cardinal satisfying $\lambda^\mu=\lambda$.
\end{itemize}

Then there exists a streamlined $\lambda^+$-Souslin tree $T$ with $V(T)=\acc(\lambda^+)$
such that $T$ admits a $\mu$-ascent path.
\end{cor}
\begin{proof}  By \cite[Theorem~3.4]{lh_lucke}, in particular, $\square^{\ind}(\lambda^+,\mu)$ holds.
Then, by \cite[Theorem~4.44]{paper23}, $\p^-(\lambda^+, \mu^{\ind},\allowbreak{\sq},1)$ holds.
By Shelah's theorem, $2^\lambda=\lambda^+$ implies $\diamondsuit(\lambda^+)$, so that, altogether $\p(\lambda^+, \mu^{\ind},\allowbreak{\sq},1)$ holds.
In addition, it is a classical theorem of Jensen that $\square_\lambda$ gives a special $\lambda^+$-Aronszajn tree,
so by Lemma~\ref{lemma39}, $\acc(\lambda^+)\in\vspec(\lambda^+)$.
It now follows from Lemma~\ref{l63} that there exists a normal and splitting streamlined $\lambda^+$-Souslin tree $T$ such that $V(T)$ covers a club in $\lambda^+$
and such that $T$ admits a $\mu$-ascent path.
Finally, the proof of Lemma~\ref{lemma33} completes this proof.
\end{proof}
\begin{remark} The conclusion of the preceding remains valid once relaxing $\square_\lambda$ to $\square_\lambda(\sq_\mu)$.
In particular, the conclusion of the preceding is compatible with $\mu$ being supercompact.
\end{remark}

We now turn to combine the preceding construction with the study of large cardinals.
The following cardinal characteristic $\chi(\kappa)$ provides a measure of how far $\kappa$ is from being weakly compact.

\begin{defn}[The $C$-sequence number of $\kappa$, \cite{paper35}]\label{defcnm}
If $\kappa$ is weakly compact, then let $\chi(\kappa):=0$. Otherwise, let
$\chi(\kappa)$ denote the least cardinal $\chi\le\kappa$
such that, for every $C$-sequence $\langle C_\beta\mid\beta<\kappa\rangle$,
there exist $\Delta\in[\kappa]^\kappa$ and $b:\kappa\rightarrow[\kappa]^{\chi}$
with $\Delta\cap\alpha\s\bigcup_{\beta\in b(\alpha)}C_\beta$
for every $\alpha<\kappa$.
\end{defn}

By \cite[Lemma~2.12(1)]{paper35}, if $\kappa$ is an inaccessible cardinal satisfying $\chi(\kappa)<\kappa$, then $\kappa$ is $\omega$-Mahlo.
The following is an expanded form of Theorem~\ref{thme}.
\begin{thm} Assuming the consistency of a weakly compact cardinal, it is consistent that for some strongly inaccessible cardinal $\kappa$
satisfying $\chi(\kappa)=\omega$, the following two hold:
\begin{itemize}
\item Every $\kappa$-Aronszajn tree admits an $\omega$-ascent path;
\item There is a $\kappa$-Souslin tree $\mathbf T$ such that $V(\mathbf T)=\acc(\kappa)$.
\end{itemize}
\end{thm}
\begin{proof} Suppose that $\kappa$ is a non-subtle weakly compact cardinal.
By possibly using a preparatory forcing, we may assume that the non-subtle weak
compactness of $\kappa$ is indestructible under forcing with $\mathrm{Add}(\kappa,1)$.
Following the proof of \cite[Theorem 3.4]{paper35},
let $\mathbb{P}$ be the standard forcing to add $\square^{\ind}(\kappa,\omega)$-sequence by closed initial segments,
let $G$ be $\mathbb{P}$-generic, and let
$\vec{\mathcal{C}}=\langle C_{\alpha,i} \mid \alpha<\kappa,~i(\alpha) \leq i<\omega \rangle$
denote the generically-added $\square^{\ind}(\kappa,\omega)$-sequence.
Work in $V[G]$.
By Clauses (1),(2) and (4) of \cite[Theorem 3.4]{paper35}, $\kappa$ is strongly inaccessible, $\chi(\kappa)=\omega$,
and every $\kappa$-Aronszajn tree admits an $\omega$-ascent path.

For every $\alpha\in \acc(\kappa)$, let
$$B_\alpha:=\{ \beta\in C_{\alpha,i(\alpha)}\mid \forall l<\omega\,[\min(C_{\alpha,i(\alpha)}\setminus\beta+1)+l\in C_{\alpha,i(\alpha)}]\}.$$

\begin{claim}\label{671} For every cofinal $B\s\kappa$, there exist $\alpha\in E^\kappa_\omega$ and $\epsilon<\alpha$
such that $(B_\alpha\setminus\epsilon)\s B$, $i(\alpha)=0$ and $\sup(\nacc(C_{\alpha,i})\cap B_\alpha)=\alpha$ for all $i<\omega$.
\end{claim}
\begin{proof} We follow the proof of \cite[Lemma~3.9]{Chris17}. Work in $V$.
For every $\alpha\in\acc(\kappa)$, let $\dot{B_\alpha}$ be the canonical $\mathbb P$-name for $B_\alpha$.
Next, let $\dot{B}$ be a $\mathbb P$-name for a cofinal subset of $\kappa$,
and let $p_0$ be an arbitrary condition in $\mathbb{P}$.
By possibly extending $p_0$, we may assume that $i(\gamma^{p_0})^{p_0}=0$.
We shall recursively define a decreasing sequence of conditions $\langle p_n \mid n<\omega \rangle$, and an increasing sequence of ordinals $\langle \beta_n \mid n<\omega \rangle$
such that for every $n<\omega$, all of the following hold:
\begin{enumerate}
\item\label{dec1} $p_{n+1}\leq p_n$;
\item $i(\gamma^{p_{n+1}})^{p_{n+1}}=0$;
\item $p_{n+1} \Vdash ``\beta_n \in \dot{B}\text{ and }\dot B_{\gamma^{p_{n+1}}}\setminus(\gamma^{p_n}+1)=\{\beta_n \}"$;
\item For every $i\le n$, $\beta_n\in\nacc(C_{\gamma^{p_{n+1}},i}^{p_{n+1}})$;
\item\label{dec4} For every $i<\omega$, $C_{\gamma^{p_{n+1}},i}^{p_{n+1}} \cap (\gamma^{p_n}+1)=C_{\gamma^{p_n},i}^{p_{n}}\cup \{\gamma^{p_n}\}$.
\end{enumerate}

Suppose $n<\omega$ is such that $\langle p_m \mid m\leq n \rangle$ and $\langle \beta_m \mid m< n \rangle$ have already been successfully defined.
Find a $p^*_n \leq p_n$ and a $\beta_n>\gamma^{p_n}$
such that $p^*_n \Vdash ``\beta_n \in \dot{B}"$.
Without loss of generality, $\gamma^{p^*_n}>\beta_n$. Now, let $\gamma:=\gamma^{p^*_n}+\omega$, so that
$$\gamma^{p_n}<\beta_n<\gamma^{p_n^*}<\gamma^{p^*_n}+\omega=\gamma.$$
Let $m <\omega$ be the least such that $m\geq \max\{n, i(\gamma^{p^*_n})^{p^*_n}\}$ and $\gamma^{p_n} \in \acc(C_{\gamma^*,m}^{p_n^*})$.
Then let $p_{n+1}$ be the unique extension of $p_n^*$ with $\gamma^{p_{n+1}}=\gamma$ and $i(\gamma)^{p_{n+1}}=0$ to satisfy the following for all $i<\omega$:
$$C_{\gamma,i}^{p_{n+1}}:=\begin{cases}
C_{\gamma^{p_n},i}^{p_{n}}\cup \{ \gamma^{p_n},\beta_n \}\cup \{ \gamma^{p_n^*}+l \mid l<\omega \},&\text{if }i\leq m;\\
C_{\gamma^{p_n^*},i}^{p^*_{n}}\cup \{ \gamma^{p_n^*}+l \mid l<\omega \},&\text{otherwise}.
\end{cases}$$

Thus, we have maintained requirements \eqref{dec1}--\eqref{dec4}.

Once completing the above recursion, we obtain a decreasing sequence of conditions $\langle p_n \mid n<\omega \rangle$.
Let $\alpha:=\sup\{ \gamma^{p_n} \mid n<\omega\}$,
and let $p$ be the unique lower bound of $\langle p_n \mid n<\omega \rangle$
to satisfy $\gamma^{p}=\alpha$, $i(\alpha)^{p}=0$, and
$C^{p}_{\alpha,i}=\bigcup_{n<\omega}C^{p_n}_{\gamma^{p_n},i}$ for every $i<\omega$.
Then $p$ is a legitimate condition satisfying $p\Vdash``\dot B_\alpha\setminus(\gamma^{p_0}+1)=\{\beta_n \mid n<\omega \} \s \dot{B}"$.
In addition, for all $i<\omega$, $\{ \beta_n \mid i \leq n<\omega\} \s \nacc(C_{\alpha,i}^{p})$.
So we are done.
\end{proof}

We claim that $\vec{\mathcal{C}}$ is a $\p^-(\kappa, \omega^{\ind},\allowbreak{\sq},1)$-sequence.
As we already know that $\vec{\mathcal{C}}$ is an $\square^{\ind}(\kappa,\omega)$-sequence,
we just need to verify that it satisfies the last bullet of Definition~\ref{indexedP} with $\theta:=1$ and $\mathcal S:=\{\kappa\}$.
But, by the same argument from the proof of \cite[Corollary~3.4]{paper23},
this boils down to showing that for every cofinal $B\s\kappa$, there exists at least one $\alpha\in\acc(\kappa)$ such that
$\sup(\nacc(C_{\alpha,i})\cap B)=\alpha$ for all $i\in[i(\alpha),\omega)$. This is covered by Claim~\ref{671}.
\begin{claim} $\diamondsuit(E^\kappa_\omega)$ holds.
\end{claim}
\begin{proof} This is a standard consequence of Claim~\ref{671} together with the fact that $\kappa^{<\kappa}=\kappa$, but we give the details.
Let $\vec X=\langle X_\beta\mid\beta<\kappa\rangle$ be a repetitive enumeration of $[\kappa]^{<\kappa}$ such that each set appears cofinally often.
Let us say that an ordinal $\alpha\in E^\kappa_\omega$ is \emph{informative} if $\sup(B_\alpha)=\alpha$ and there are $\epsilon<\kappa$ and a subset $A_\alpha\s \alpha$ such that
$A_\alpha\cap\gamma=X_\beta\cap\gamma$ for every pair $\gamma<\beta$ of ordinals from $B_\alpha\setminus\epsilon$.
Note that if $\alpha$ is informative, then the set $A_\alpha$ is uniquely determined.
For a noninformative $\alpha\in E^\kappa_\omega$, we let $A_\alpha:=\emptyset$.

To verify that $\langle A_\alpha\mid \alpha\in E^\kappa_\omega\rangle$ witnesses $\diamondsuit(E^\kappa_\omega)$,
let $A$ be a subset of $\kappa$ and let $C$ be a club in $\kappa$, and we shall find an $\alpha\in C\cap E^\kappa_\omega$ such that $A\cap\alpha=A_\alpha$.

By the choice of $\vec X$, we may fix a strictly increasing function $f:\kappa\rightarrow\kappa$
satisfying that $A\cap\xi=X_{f(\xi)}$ for every $\xi<\kappa$.
Consider the club $D:=\{\delta\in C\mid f[\delta]\s\delta\}$.
Let $B$ be some cofinal subset of $\im(f)$ sparse enough to satisfy that for every pair $\gamma<\beta$ of ordinals from $B$,
there exists a $\delta\in D$ with $\gamma<\delta<\beta$.
Using Claim~\ref{671}, fix $\alpha\in E^\kappa_\omega$ and $\epsilon<\alpha$
such that $(B_\alpha\setminus\epsilon)\s B$ and $\sup(B_\alpha)=\alpha$.
Now, let $\gamma<\beta$ be a pair of ordinals in $B_\alpha\setminus\epsilon$.
As $\gamma,\beta\in B$, we may pick a $\delta\in D$ with $\gamma<\delta<\beta$.
As $\beta\in B\s\im(f)$, we may also pick a $\xi<\kappa$ such that $\beta=f(\xi)$. Since $f[\delta]\s\delta\s \beta$, it must be the case that $\xi\ge\delta>\gamma$.
So $A\cap\gamma=(A\cap\xi)\cap\gamma=X_\beta\cap\gamma$.
Thus, we showed that $A\cap\gamma=X_\beta\cap\gamma$ for every pair $\gamma<\beta$ of ordinals in $B_\alpha\setminus\epsilon$, and hence $\alpha$ is informative and $A_\alpha=A\cap\alpha$.
In addition, for every pair $\gamma<\beta$ of ordinals in $B_\alpha\setminus\epsilon$,
there exists $\delta\in D$ with $\gamma<\delta<\beta$, and hence $\alpha\in\acc^+(D)\s C$.
\end{proof}

Altogether, $\p(\kappa, \omega^{\ind},\allowbreak{\sq},1)$ holds.
Since $\kappa$ is a strongly inaccessible cardinal that is non-subtle, Corollary~\ref{stp} implies that there exists a streamlined $\kappa$-tree $K$ such that $V^-(K)$ covers a club in $\kappa$.
So by appealing to Lemma~\ref{l63} and then to Lemma~\ref{lemma33},
we infer that there exists a $\kappa$-Souslin tree $\mathbf T$ with $V(\mathbf T)=\acc(\kappa)$.
\end{proof}

By \cite[Theorem~2.30]{paper48}, $\chi(\kappa)=0$ refutes $\clubsuit_{\ad}(\reg(\kappa))$.
An easy variant of that proof yields that $\chi(\kappa)=0$ furthermore refutes $\clubsuit_{\ad}(\reg(\kappa)\cap D)$ for every club $D\s\kappa$.
It follows from the preceding theorem together with the proof of \cite[Theorem~2.23]{paper48} that $\chi(\kappa)=\omega$ is
compatible with $\clubsuit_{\ad}(D)$ holding for some club $D\s\kappa$.
Whether this can be improved to $\chi(\kappa)=1$ remains an open problem.

\renewcommand{\thesection}{A}
\section{A new sufficient condition for a Dowker space}\label{secA}
\begin{defn}[\cite{paper48}]\label{clubad}  Let $\mathcal S$ be a collection of stationary subsets of a regular uncountable cardinal $\kappa$,
and $\mu,\theta$ be nonzero cardinals below $\kappa$.
The principle
$\clubsuit_{\ad}(\mathcal S,\mu,\theta)$ asserts the existence of a sequence $\langle \mathcal A_\alpha\mid \alpha\in\bigcup\mathcal S\rangle$ such that:
\begin{enumerate}
\item For every $\alpha\in\acc(\kappa)\cap\bigcup\mathcal S$, $\mathcal A_\alpha$ is a pairwise disjoint family of $\mu$ many cofinal subsets of $\alpha$;
\item For every $\mathcal B\s[\kappa]^\kappa$ of size $\theta$, for every $S\in\mathcal S$, there are stationarily many $\alpha\in S$ such that $\sup(A\cap B)=\alpha$ for all $A\in\mathcal A_\alpha$ and $B\in\mathcal B$;\footnote{Note that the existence
of stationarily many such $\alpha\in S$ is no stronger than the existence of just one $\alpha\in S$. See \cite[Corollary~3.4]{paper23} for the prototype argument.}
\item For all $A\neq A'$ from $\bigcup_{S\in\mathcal S}\bigcup_{\alpha\in S}\mathcal A_\alpha$, $\sup(A\cap A')<\sup(A)$.
\end{enumerate}
\end{defn}
\begin{remark} The variation $\clubsuit_{\ad}(\mathcal S,\mu,{<}\theta)$ asserts the existence of a sequence simultaneously
witnessing $\clubsuit_{\ad}(\mathcal S,\mu,\vartheta)$ for all $\vartheta<\theta$.
\end{remark}

By \cite[Lemma~2.10]{paper48},
for a pair $\chi<\kappa$ of infinite regular cardinals,
for a stationary subset $S$ of $E^\kappa_\chi$,
Ostaszewski's principle $\clubsuit(S)$ implies $\clubsuit_{\ad}(\mathcal S,\chi,{<}\omega)$ for some partition $\mathcal S$ of $S$ into $\kappa$ many stationary sets.
The next theorem reduces the hypothesis ``$S\s E^\kappa_\chi$'' down to ``$S\cap\tr(S)=\emptyset$''.

\begin{lemma} \label{lem7.1}
Suppose:
\begin{itemize}
\item $\mu,\theta<\kappa=\kappa^{<\theta}$ are infinite cardinals;
\item $S\s E^\kappa_{\ge\max\{\mu,\theta\}}$ is stationary and $\tr(S)\cap S=\emptyset$;
\item $\clubsuit(S)$ holds.
\end{itemize}

Then $\clubsuit_{\ad}(\mathcal S,\mu,{<}\theta)$ holds
for some partition $\mathcal S$ of $S$ into $\kappa$ many stationary sets. More generally,
for every $Z\s\kappa$ such that $S\s\acc^+(Z)$,
there exists a matrix $\langle A_{\delta,i}\mid\delta\in S, i<\mu\rangle$
and a partition $\mathcal S$ of $S$ into $\kappa$ many pairwise disjoint stationary sets such that:
\begin{enumerate}
\item For all $\delta\in S$, $\langle A_{\delta,i}\mid i<\mu\rangle$ is a sequence of pairwise disjoint subsets of $Z\cap\delta$, and $\sup(A_{\delta,i})=\delta$;
\item For every $(\gamma,\delta)\in[S]^2$, for all $i,j<\mu$, $\sup(A_{\gamma,i}\cap A_{\delta,j})<\gamma$;
\item For every $\vartheta<\theta$, every sequence $\langle B_\tau\mid \tau<\vartheta\rangle$ of cofinal subsets of $Z$ and every $S'\in\mathcal S$, there exists $\delta\in S'$ such that $\sup(A_{\delta,i}\cap B_\tau)=\delta$
for all $i<\mu$ and $\tau<\vartheta$.
\end{enumerate}
\end{lemma}
\begin{proof}  By \cite[Theorem~3.7]{paper23}, since $\clubsuit(S)$ holds,
we may find a partition $\langle S_{\vartheta,\iota}\mid \vartheta<\theta, \iota<\kappa\rangle$ of $S$
into stationary sets
such that $\clubsuit(S_{\vartheta,\iota})$ holds for all $\vartheta<\theta$ and $\iota<\kappa$.
For all $\vartheta<\theta$ and $\iota<\kappa$, since $\clubsuit(S_{\vartheta,\iota})$ holds and $\kappa^{\vartheta}=\kappa$,
by \cite[Lemma~3.5]{paper48}, we may fix a matrix $\langle X_\delta^\tau\mid \delta\in S_{\vartheta,\iota}, \tau<\vartheta\rangle$ such that,
for every sequence $\langle X^\tau\mid \tau<\vartheta\rangle$ of cofinal subsets of $\kappa$,
there are stationarily many $\delta\in S_{\vartheta,\iota}$, such that, for all $\tau<\vartheta$,
$X^\tau_\delta\s X^\tau\cap\delta$ and $\sup(X^\tau_\delta)=\delta$.

Now, let $Z\s\kappa$ with $S\s\acc^+(Z)$ be given.
For all $\vartheta<\theta$, $\iota<\kappa$, $\delta\in S_{\vartheta,\iota}$ and $\tau<\vartheta$, we do the following:
\begin{itemize}
\item if $X^\tau_\delta\cap Z$ is a cofinal subset of $\delta$,
then let $Y^\tau_\delta:=X^\tau_\delta\cap Z$.
Otherwise, let $Y^\tau_\delta$ be an arbitrary cofinal subset of $Z\cap\delta$;
\item since $\delta\in S\s\kappa\setminus\tr(S)$, we may fix a club $C_\delta\s\delta$ disjoint from $S$,
and then, by \cite[Lemma~3.3]{paper23}, we may find a cofinal subset $Z^\tau_\delta$ of $Y^\tau_\delta$
such that in-between any two points of $Z^\tau_\delta$ there exists a point of $C_\delta$,
so that $\acc^+(Z^\tau_\delta)\cap S=\emptyset$.
\end{itemize}

As $\cf(\delta)\ge\theta>\vartheta$ and by possibly thinning out, we may assume that $\langle Z^\tau_\delta\mid \tau<\vartheta\rangle$ consists of pairwise disjoint cofinal subsets of $Z\cap\delta$.
As $\cf(\delta)\ge\mu$, for every $\tau<\vartheta$, we may fix a partition $\langle Z_\delta^{\tau,i}\mid i<\mu\rangle$ of $Z_\delta^{\tau}$ into cofinal subsets of $\delta$.
For every $i<\mu$, let $$A_{\delta,i}:=\bigcup_{\tau<\vartheta}Z_\delta^{\tau,i}.$$

For every $i<\mu$,
since $\acc^+(Z_\delta^{\tau,i})\cap S\s\acc^+(Z^\tau_\delta)\cap S=\emptyset$ ,
and since $\delta\in S\s E^\kappa_{>\vartheta}$, we get that $\acc^+(A_{\delta,i})\cap S=\emptyset$.
So $\langle A_{\delta,i}\mid i<\mu\rangle$ is a sequence of pairwise disjoint cofinal subsets of $\delta$,
and for every $\gamma\in S\cap\delta$ and every cofinal subset $A\s\gamma$, $\sup(A\cap A_{\delta,i})<\gamma$.
Thus, we have already taken care of Clauses (1) and (2).

Next, consider $\mathcal S:=\{ \bigcup_{\vartheta<\theta}S_{\vartheta,\iota}\mid \iota<\kappa\}$ which is a partition of $S$ into $\kappa$ many stationary sets.
Now, given $\vartheta<\theta$, a sequence $\langle B_\tau\mid \tau<\vartheta\rangle$ of cofinal subsets of $Z$, and some $S'\in\mathcal S$,
we may find $\iota<\kappa$ such that $S'\supseteq S_{\vartheta,\iota}$,
and find $\delta\in S_{\vartheta,\iota}$ such that, for all $\tau<\vartheta$,
$X^\tau_\delta\s B_\tau\cap\delta$ and $\sup(X^\tau_\delta)=\delta$.
In particular, for all $\tau<\vartheta$ and $i<\mu$, $Z_\delta^{\tau,i}\s Z^\tau_\delta\s Y^\tau_\delta=X^\tau_\delta\cap Z\s B_\tau$.
Therefore, for all $\tau<\vartheta$ and $i<\mu$, $\sup(A_{\delta,i}\cap B_\tau)=\delta$.
\end{proof}

\begin{cor}\label{cor84}
Suppose that $\clubsuit(S)$ holds for some nonreflecting stationary subset $S$ of $\kappa$.
Then $\clubsuit_{\ad}(\mathcal S,\omega,{<}\omega)$ holds
for some partition $\mathcal S$ of $S$ into $\kappa$ many stationary sets.\qed
\end{cor}

The preceding yields the proof of Theorem~\ref{thmf}
which in turn extends an old result of Good \cite{MR1216813} who got a Dowker space of size $\lambda^+$
from $\clubsuit(S)$ holding over a nonreflecting stationary $S\s E^{\lambda^+}_\omega$.\footnote{Strictly speaking, the hypothesis in \cite{MR1216813} is $\clubsuit_{\lambda^+}(S,2)$, but \cite[Lemma~3.5]{paper23} shows that this is no stronger than the vanilla $\clubsuit(S)$.}

\begin{cor}  If $\clubsuit(S)$ holds over a nonreflecting stationary $S\s\kappa$, then
there are $2^\kappa$ many pairwise nonhomeomorphic Dowker spaces of size $\kappa$.
\end{cor}
\begin{proof} By \cite[Theorem~A.1]{paper54}, if $\clubsuit_{\ad}(\mathcal S,1,2)$ holds
for a partition $\mathcal S$ of a nonreflecting stationary subset of $\kappa$
into $\kappa$ many stationary sets,
then there are $2^\kappa$ many pairwise nonhomeomorphic Dowker spaces of size $\kappa$.
\end{proof}

Our last corollary deals with the problem of having $\clubsuit_{\ad}$ hold over a club subset of a successor cardinal.
\begin{cor}
Suppose that $\kappa=\lambda^+$ for some infinite cardinal $\lambda$,
and that $\clubsuit(E^\kappa_\theta)$ holds for every $\theta\in\reg(\kappa)$.
Then there exists a partition $\mathcal S$ of some club subset $D\s\acc(\kappa)$
into $\kappa$ many sets such that $\clubsuit_{\ad}(\mathcal S,\omega,1)$ holds.
Furthermore,  there is a matrix $\langle A_{\delta,i}\mid \delta\in D, i<\cf(\delta)\rangle$ such that:
\begin{enumerate}
\item For every $\delta\in D$, $\langle A_{\delta,i}\mid i<\cf(\delta)\rangle$ is sequence of pairwise disjoint cofinal subsets of $\delta$;
\item For all $A\neq A'$ from $\{ A_{\delta,i}\mid \delta\in D, i<\cf(\delta)\}$ , $\sup(A\cap A')<\sup(A)$;
\item For every cofinal $B\s\kappa$, for every $S\in\mathcal S$,
there are stationarily many $\delta\in S$ such that $\sup(A_{\delta,i}\cap B)=\delta$ for all $i<\cf(\delta)$.
\end{enumerate}
\end{cor}
\begin{proof} Let $\langle Z_\mu\mid \mu\in\reg(\kappa)\rangle$ be a partition of $\kappa$ into cofinal sets.
Let $D:=\bigcap_{\mu\in\reg(\kappa)}\acc^+(Z_\mu)$.
For every $\mu\in\reg(\kappa)$, by Lemma~\ref{lem7.1},
we may fix a matrix $\langle A_{\delta,i}\mid\delta\in E^\kappa_\mu, i<\mu\rangle$
and a partition $\langle S_{\mu,\iota}\mid \iota<\kappa\rangle$ of $E^\kappa_\mu$ into $\kappa$ many pairwise disjoint stationary sets such that:
\begin{itemize}
\item For all $\delta\in E^\kappa_\mu$, $\langle A_{\delta,i}\mid i<\mu\rangle$ is a sequence of pairwise disjoint subsets of $Z_\mu\cap\delta$, and $\sup(A_{\delta,i})=\delta$;
\item For every $(\gamma,\delta)\in[E^\kappa_\mu]^2$, for all $i,j<\mu$, $\sup(A_{\gamma,i}\cap A_{\delta,j})<\gamma$;
\item For every cofinal $B\s Z_\mu$, for every $\iota<\kappa$, there exists $\delta\in S_{\mu,\iota}$ such that $\sup(A_{\delta,i}\cap B)=\delta$ for all $i<\mu$.
\end{itemize}

Putting these matricies together, we get a matrix $\langle A_{\delta,i}\mid\delta\in D, i<\cf(\delta)\rangle$ satisfying Clause~(1).
In addition, since $Z_\mu\cap Z_{\mu'}=\emptyset$ for $\mu\neq\mu'$, Clause~(2) is satisfied.
Now, $\mathcal S:=\{ \bigcup_{\mu\in\reg(\kappa)}S_{\mu,\iota}\mid \iota<\kappa\}$
is a partition of $D$ into $\kappa$ many stationary sets.
By the pigeonhole principle, for every cofinal $B\s\kappa$, there exists some $\mu\in\reg(\kappa)$
such that $B\cap Z_\mu$ is cofinal in $\kappa$.
So, for every $S\in\mathcal S$,
there exist $\iota<\kappa$ and $\delta\in S_{\mu,\iota}\s S$ such that $\sup(A_{\delta,i}\cap B)=\delta$ for all $i<\cf(\delta)$.
\end{proof}

\section*{Acknowledgments}

The first and second author were supported by the Israel Science Foundation (grant agreement 203/22).
The first and third author were supported by the European Research Council (grant agreement ERC-2018-StG 802756).

Some of the results of this paper were presented by the second author in a poster session at the \emph{Young Set Theory Workshop} in Novi Sad, August 2022.
Additional results were presented by the first author as part of a
graduate course on \emph{Set Theory, Algebra and Analysis} at the Fields Institute for Research in Mathematical Sciences during the spring semester of 2023.
We thank the corresponding organizers for the opportunity to present this work and the participants for their
feedback.

\end{document}